\documentclass[letterpaper,11pt]{article}
\usepackage{geometry}
\usepackage{amsmath,amsthm,amsfonts,amssymb,bm,bbm,mathtools}
\usepackage{algorithm,algorithmic}
\usepackage{authblk}
\usepackage[numbers]{natbib}
\usepackage{listings}
\usepackage{graphicx,hyperref, footnote, comment}
\usepackage{xcolor}
\usepackage{array,booktabs,tabularx}

\newcolumntype{Y}{>{\raggedright\arraybackslash}X}

\newcommand{\parenth}[1]{\left( #1 \right)}
\newcommand{\norm}[1]{\left\| #1 \right\|}

\newcommand{\braces}[1]{\left\{ #1 \right \}}

\newcommand{\angles}[1]{\left\langle #1 \right \rangle}

\newcommand{\matr}[1]{\mathbf{#1}}
\newcommand{\vect}[1]{\mathbf{#1}}

\newtheorem{theorem}{Theorem}
\newtheorem{proposition}{Proposition}
\newtheorem{lemma}{Lemma}

\newlength{\widebarargwidth}
\newlength{\widebarargheight}
\newlength{\widebarargdepth}

\makeatletter
\long\def\@makecaption#1#2{
        \vskip 0.8ex
        \setbox\@tempboxa\hbox{\small {\bf #1:} #2}
        \parindent 1.5em
        \dimen0=\hsize
        \advance\dimen0 by -3em
        \ifdim \wd\@tempboxa >\dimen0
                \hbox to \hsize{
                        \parindent 0em
                        \hfil
                        \parbox{\dimen0}{\def\baselinestretch{0.96}\small
                                {\bf #1.} #2
                                }
                        \hfil}
        \else \hbox to \hsize{\hfil \box\@tempboxa \hfil}
        \fi
        }
\makeatother

\long\def\comment#1{}
\definecolor{battleshipgrey}{rgb}{0.52, 0.52, 0.51}
\definecolor{darkgray}{rgb}{0.66, 0.66, 0.66}
\definecolor{darkgreen}{rgb}{0.0, 0.2, 0.13}
\definecolor{darkspringgreen}{rgb}{0.09, 0.45, 0.27}
\definecolor{dukeblue}{rgb}{0.0, 0.0, 0.61}
\definecolor{olivedrab7}{rgb}{0.24, 0.2, 0.12}
\definecolor{darkblue}{rgb}{0.0, 0.0, 0.55}
\definecolor{darkscarlet}{rgb}{0.34, 0.01, 0.1}
\definecolor{candyapplered}{rgb}{1.0, 0.03, 0.0}
\definecolor{ao(english)}{rgb}{0.0, 0.5, 0.0}
\definecolor{applegreen}{rgb}{0.55, 0.71, 0.0}
\definecolor{calorange}{HTML}{FDB515}
\definecolor{calblue}{HTML}{003262}

\usepackage[capitalise,nameinlink]{cleveref}
\crefname{lemma}{Lemma}{Lemmas}
\crefname{fact}{Fact}{Facts}
\crefname{theorem}{Theorem}{Theorems}
\crefname{corollary}{Corollary}{Corollaries}
\crefname{claim}{Claim}{Claims}
\crefname{example}{Example}{Examples}
\crefname{problem}{Problem}{Problems}
\crefname{setting}{Setting}{Settings}
\crefname{definition}{Definition}{Definitions}
\crefname{assumption}{Assumption}{Assumptions}
\crefname{subsection}{Subsection}{Subsections}
\crefname{section}{Section}{Sections}
\crefname{figure}{Figure}{Figures}

\hypersetup{
  colorlinks=true,
  citecolor=calblue,
  linkcolor=calblue,
  urlcolor=calblue
}

\title{A Unifying View of Anchoring via Operator-Side Tikhonov Regularization}
\author{Zihao Chen\thanks{\texttt{zihaochen@berkeley.edu}. This work was done while the author was at UC Berkeley.}}
\date{}

\begin{document}
\maketitle

\begin{abstract}
Anchored fixed point and monotone equation methods, including Halpern iteration, extra anchored gradient, and their relatives, add a vanishing pull toward a reference point to obtain last-iterate guarantees. Existing anchored variants often achieve sharp last-iterate guarantees, but from the update-level perspective the placement of the anchor can be algorithm-specific and conceptually opaque. We show that anchoring admits a single operator-side construction: regularize the operator queried by the base method with a vanishing Tikhonov term, then run the unmodified base method. Applied to the Picard iteration, this recipe reproduces the Halpern iteration; applied to the forward step, extragradient (EG), and past extragradient (PEG, also known as Popov's method), it yields three variants whose anchor placements inherit the base method's query pattern. The forward-step instantiation gives a new residual convergence guarantee, while the EG and PEG instantiations give new regularized variants. The four analyses share a residual recurrence, recovering the \(O(1/k)\) Halpern residual-norm convergence rate, giving \(O(1/\sqrt{k})\) for the regularized forward step, and giving \(O(1/k)\) for the regularized EG and PEG variants in the unconstrained monotone Lipschitz setting.
\end{abstract}

\section{Introduction}
\label{sec:introduction}

Many equilibrium conditions and first-order convex optimality conditions can be
written as fixed point or monotone equation problems. In a fixed point problem,
one seeks \(\vect{x}^\star\) such that
\(T(\vect{x}^\star)=\vect{x}^\star\). In a monotone equation problem, one seeks
\(\vect{x}^\star\) such that
\(F(\vect{x}^\star)=\vect{0}\), where \(F\) is monotone in the standard sense
made precise in \cref{sec:setup}. For instance, the first-order conditions of a smooth
unconstrained convex-concave minmax problem give a monotone saddle operator.
In this work, we focus on last-iterate residual convergence, measured by
\[
    \norm{T(\vect{x}_k)-\vect{x}_k},
    \qquad
    \norm{F(\vect{x}_k)}.
\]

The simplest fixed point method is Picard iteration,
\(\vect{x}_{k+1}=T(\vect{x}_k)\). When \(T\) is a contraction, the usual
contraction argument gives convergence; for a merely nonexpansive \(T\), that
argument no longer applies. The corresponding method for
\(F(\vect{x})=\vect{0}\) is the forward step
\(\vect{x}_{k+1}=\vect{x}_k-\eta_k F(\vect{x}_k)\). When
\(F=\nabla f\) for a smooth convex function \(f\), this is gradient descent.
It behaves well, with suitable step sizes, under stronger assumptions such as
cocoercivity, but for a merely monotone Lipschitz operator it can diverge.

Anchoring repairs or sharpens this last-iterate behavior by
adding a vanishing pull toward a reference point \(\vect{x}_0\). Halpern
iteration does this for nonexpansive fixed point problems
\cite{halpern1967fixed,browder1967convergence,wittmann1992approximation}. Its Tikhonov interpretation is
direct: the regularization makes the current problem better conditioned, and
the regularization is then slowly removed. For extragradient (EG) and past
extragradient (PEG, also known as Popov's method), the picture is less clear. These
methods already stabilize monotone Lipschitz problems by changing the query
pattern and achieve last-iterate convergence
\cite{korpelevich1976extragradient,popov1980modification,gorbunov2022extragradient,gorbunov2022last}.
Anchored variants of these methods improve the last-iterate residual rate in related settings
\cite{yoon2021accelerated,lee2021fast,boct2024extra,tran2021halpern,bot2025fast,tran2025accelerated,alcala2025stochastic}.
These guarantees are obtained through method-specific analyses and sometimes carefully
chosen Lyapunov functions. This leaves a basic structural question about the
mechanism of acceleration: can the anchor placement be derived systematically
from the query pattern of the base method?

This paper proposes to regularize the operator being queried,
then run the original template method unchanged. For fixed point problems, the
regularized map is
\[
    T_k(\vect{x})
    =
    (1-\delta_k)T(\vect{x})+\delta_k\vect{x}_0.
\]
For monotone equations, the regularized object is the scaled displacement
queried by the method:
\[
    G_k(\vect{x})
    =
    \eta_kF(\vect{x})+\delta_k(\vect{x}-\vect{x}_0).
\]
The anchor terms that appear after expanding the update are then determined by
the query points of the base method.

The fixed point case gives the first example. Applying Picard iteration to
\(T_k\) gives
\[
    \vect{x}_{k+1}
    =
    T_k(\vect{x}_k)
    =
    (1-\delta_k)T(\vect{x}_k)+\delta_k\vect{x}_0,
\]
which is Halpern iteration. As another example, extragradient shows the anchor-placement mechanism
more visibly. With a fixed \(\eta\), applying the two-query template to
\(G_k\) gives
\[
    \vect{x}_{k+1/2}
    =
    \vect{x}_k-G_k(\vect{x}_k),
    \qquad
    \vect{x}_{k+1}
    =
    \vect{x}_k-G_k(\vect{x}_{k+1/2}).
\]
Expanding the second line yields
\[
    \vect{x}_{k+1}
    =
    \vect{x}_k-\eta F(\vect{x}_{k+1/2})
    -\delta_k(\vect{x}_{k+1/2}-\vect{x}_0).
\]
Hence, under this operator-side rule, the second anchor term is evaluated at the
lookahead point because that is where the second operator query is made.

\paragraph{Contributions.}
Existing anchored extragradient and Popov-type methods achieve accelerated
last-iterate residual rates
\cite{yoon2021accelerated,lee2021fast,boct2024extra,tran2021halpern,bot2025fast}. Instead of explaining why each of those particular anchored
updates achieves acceleration, we give a separate
operator-side route to anchored algorithms. The rule is simple: regularize the
operator (or its scaled displacement) with a vanishing Tikhonov term, and run
the unmodified base template on the result. Applied to Picard, the rule
recovers Halpern iteration and the known \(O(R/k)\) residual bound. Applied
to the forward step, it stabilizes the method and gives an
\(O(LR/\sqrt{k})\) residual-norm guarantee under monotonicity and Lipschitz
continuity, the first such guarantee without extragradient or the Popov-type
past query.\footnote{The regularized forward-step result first appeared in the
author's dissertation \cite{chen2025algorithmic}. It was later independently
discovered via AI-assisted formal proof search in Lean by Surina et
al.~\cite{surina2026improved}.} Applied to EG and
PEG, it gives the new regularized variants Reg-EG and Reg-PEG, both with
\(O(LR/k)\) residual-norm bounds. The four analyses then
reduce to the same recurrence. This perspective results in a single intuitive derivation that
produces new accelerated anchored variants and their analyses uniformly.

\paragraph{Organization.}
\Cref{sec:setup} introduces the setup, base methods, and related work.
\Cref{sec:operator-side-regularization} discusses the operator-side
regularization principle and uses Halpern iteration as the fixed point
prototype. \Cref{sec:main-results} establishes the monotone equation results
for the regularized forward step, Reg-EG, and Reg-PEG.
\Cref{sec:conclusion} concludes.

\section{Background and Preliminaries}
\label{sec:setup}

\subsection{Problem Setup}
\label{sec:problems}

We use two operator templates throughout the paper. The fixed point results
study a map \(T\) and certify an iterate by its fixed point residual, while the
monotone equation results study an operator \(F\) and certify an iterate by the
norm of \(F(\vect{x})\). The notation differs slightly across the two settings, but the
role of the residual is the same: it is the last-iterate certificate controlled
by the regularized dynamics.

For fixed point problems, we work in a general normed space
\((\mathcal X,\norm{\cdot})\) with a map \(T:\mathcal X\to\mathcal X\). The map
is \emph{nonexpansive} if
\[
    \norm{T(\vect{x})-T(\vect{y})}
    \le
    \norm{\vect{x}-\vect{y}}
    \qquad
    \text{for all }\vect{x},\vect{y}\in\mathcal X.
\]
It is a \emph{contraction} if the right-hand side can be multiplied by some
\(\gamma<1\). For monotone equations, we work in a real Hilbert space
\((\mathcal H,\angles{\cdot,\cdot})\) with a single-valued operator
\(F:\mathcal H\to\mathcal H\). The operator is \emph{monotone} if
\[
    \angles{F(\vect{x})-F(\vect{y}),\vect{x}-\vect{y}}\ge0
    \qquad
    \text{for all }\vect{x},\vect{y}\in\mathcal H.
\]
We use two standard quantitative refinements: \(F\) is
\emph{\(\mu\)-strongly monotone} if the left-hand side is at least
\(\mu\norm{\vect{x}-\vect{y}}^2\), and it is \emph{\(L\)-Lipschitz} if
\[
    \norm{F(\vect{x})-F(\vect{y})}
    \le
    L\norm{\vect{x}-\vect{y}}.
\]
The operator \(F\) is \emph{\(\beta\)-cocoercive} if
\[
    \angles{F(\vect{x})-F(\vect{y}),\vect{x}-\vect{y}}
    \ge
    \beta\norm{F(\vect{x})-F(\vect{y})}^2.
\]
Cocoercivity is stronger than monotonicity and gives a direct bridge between
the two settings: if \(F\) is \(\beta\)-cocoercive, then \(I-\eta F\) is
nonexpansive for \(0<\eta\le2\beta\). The main monotone equation results below
do \textit{not} assume cocoercivity.

The corresponding solution sets are
\[
    \operatorname{Fix}(T)
    =
    \braces{\vect{x}\in\mathcal X:T(\vect{x})=\vect{x}},
    \qquad
    \operatorname{Zer}(F)
    =
    \braces{\vect{x}\in\mathcal H:F(\vect{x})=\vect{0}}.
\]
Every theorem assumes the relevant solution set is nonempty. We fix an anchor
\(\vect{x}_0\), choose a reference solution \(\vect{x}^\star\), and write
\[
    R\coloneqq\norm{\vect{x}_0-\vect{x}^\star}.
\]
All bounds are stated relative to this anchored radius.

With this notation in place, the residuals reported throughout are
\[
    r_T(\vect{x})\coloneqq\norm{T(\vect{x})-\vect{x}},
    \qquad
    r_F(\vect{x})\coloneqq\norm{F(\vect{x})}.
\]
The auxiliary residual \(\norm{G_k(\vect{x})}\) used in the
analyses is introduced with \(G_k\) in \cref{sec:framework_principle}.

A standard machine-learning instance of the monotone equation template is the
saddle-point operator. Given a smooth convex-concave function
\(\mathcal L(\vect{u},\vect{v})\), define
\[
    F(\vect{u},\vect{v})
    =\bigl(\nabla_{\vect{u}}\mathcal L(\vect{u},\vect{v}),
           -\nabla_{\vect{v}}\mathcal L(\vect{u},\vect{v})\bigr).
\]
This operator is monotone, and its zeros are the saddle points of
\(\mathcal L\). This is the setting most commonly studied for EG and
PEG-type methods.

\subsection{Base Methods}

The operator-side construction starts from the query pattern of a base method.
For this reason, we record both the updates and the points where the operator is
evaluated.

\paragraph{Picard iteration.}
For a map \(T\), Picard iteration is
\begin{equation}\label{eq:picard_iteration}
    \vect{x}_{k+1}=T(\vect{x}_k).
\end{equation}
When \(T\) is a contraction, the Banach fixed point theorem gives convergence to
the unique fixed point. \Cref{sec:operator-side-regularization} applies this
template to the regularized map \(T_k\), which recovers Halpern iteration.

\paragraph{Forward step.}
The explicit one-call update for solving \(F(\vect{x})=\vect{0}\) is
\begin{equation}\label{eq:forward_step_std}
    \vect{x}_{k+1}
    =
    \vect{x}_k-\eta_kF(\vect{x}_k).
\end{equation}
In optimization notation this is gradient descent when \(F=\nabla f\), and
gradient descent-ascent when \(F\) is a saddle operator. For merely monotone
Lipschitz operators, the forward step need not be stable.
\Cref{sec:main-results} shows that adding the operator-side Tikhonov term
alone is enough to stabilize it.

\paragraph{Extragradient.}
The extragradient (EG) template uses one lookahead query and one correction
query:
\begin{equation}\label{eq:eg_std}
    \vect{x}_{k+1/2}
    =
    \vect{x}_k-\eta F(\vect{x}_k),
    \qquad
    \vect{x}_{k+1}
    =
    \vect{x}_k-\eta F(\vect{x}_{k+1/2}).
\end{equation}
\Cref{sec:main-results} applies the operator-side rule to this two-query
pattern; the induced Reg-EG variant places the second anchor at the lookahead
point and obtains an \(O(LR/k)\) residual-norm bound.

\paragraph{Past extragradient.}
Past extragradient (PEG), also known as Popov's method, can be written with an
auxiliary sequence as
\begin{equation}\label{eq:og_std}
    \tilde{\vect{x}}_k
    =
    \vect{x}_k-\eta F(\tilde{\vect{x}}_{k-1}),
    \qquad
    \vect{x}_{k+1}
    =
    \vect{x}_k-\eta F(\tilde{\vect{x}}_k).
\end{equation}
Below we use the initialization \(\tilde{\vect{x}}_{-1}=\vect{x}_0\). In
unconstrained Euclidean saddle-point problems, the method is also called
optimistic gradient descent-ascent (OGDA).
\Cref{sec:main-results} applies the same rule to the PEG query pattern,
producing Reg-PEG with an \(O(LR/k)\) residual-norm bound, matching the known
rates of nearby anchored PEG/OGDA variants.

\subsection{Related Work}

\paragraph{Fixed-point regularization and Halpern iteration.}
Halpern iteration and Browder approximants are classical ways to stabilize
nonexpansive fixed-point dynamics and select fixed points
\cite{browder1967convergence,halpern1967fixed,wittmann1992approximation}.
Viscosity approximation replaces the fixed anchor by a contraction, and
Tikhonov--Mann variants are part of the same fixed-point regularization picture
\cite{moudafi2000viscosity,xu2004viscosity,cheval2023modified}. Quantitative
analyses often describe bounds on \(\norm{\vect{x}_k-T(\vect{x}_k)}\) as rates
of \(T\)-asymptotic regularity; here these are fixed-point residual bounds,
including the known \(O(1/k)\) Halpern behavior
\cite{leustean2007rates,kohlenbach2011quantitative,sabach2017first,lieder2021convergence,park2022exact,contreras2023optimal}.
Tikhonov-regularized monotone flows also have fixed-point and Halpern
specializations \cite{bot2024tikhonov}. In this paper, the fixed-point result
is a known theorem; we use it as the simplest instance of the operator-side
tracking argument and provide a new analysis in a general normed space.

\paragraph{Tikhonov regularization for variational inequalities and monotone inclusions.}
Tikhonov regularization goes back to \cite{Tikhonov1943}, and related
regularization ideas are classical in monotone operator theory. Proximal-point
and resolvent regularization go back to
\cite{martinet1970regularisation,rockafellar1976monotone}; those methods solve
regularized resolvent or proximal subproblems, often centered at the current
iterate. A distinct iterative Tikhonov line adds a vanishing strongly monotone
term, often relative to a fixed anchor or to the origin, in variational
inequality, inclusion, Nash-equilibrium, and stochastic-approximation schemes
\cite{browder1966existence,xu2010regularization,kannan2012distributed,koshal2013regularized,shanbhag2013stochastic}.
Recent concurrent anchored-gradient analyses study closely related one-call updates in
smooth convex-concave and monotone-inclusion settings
\cite{surina2026improved,cai2026last}.

\paragraph{Extragradient and past extragradient.}
EG and PEG methods (also known as optimistic gradient descent-ascent
methods in min-max optimization) stabilize monotone Lipschitz equations where
the plain forward step (gradient descent-ascent in min-max optimization) can fail
\cite{korpelevich1976extragradient,popov1980modification}. Modern unanchored
last-iterate analyses give \(O(1/k)\) bounds on squared residuals
\cite{gorbunov2022extragradient,gorbunov2022last}. In the residual norm used in
this paper, that is an \(O(1/\sqrt{k})\) rate. The Reg-EG and Reg-PEG theorems
below are stated directly as \(O(LR/k)\) residual-norm bounds.

\paragraph{Anchored EG and PEG variants.}
Anchored EG-family variants, including EAG, FEG, flexible anchored G-EAG, and
broader accelerated EG-type frameworks, establish improved guarantees under
their own assumptions and for their own residual notions
\cite{yoon2021accelerated,lee2021fast,boct2024extra,tran2025accelerated}.
Halpern-type splitting, anchored Popov/APV variants, fast OGDA, and
moving-anchor Popov variants provide related anchored or Tikhonov mechanisms on
the PEG side
\cite{diakonikolas2020halpern,tran2021halpern,bot2025fast,alcala2025stochastic}.
These works are the closest algorithmic neighbors of the present paper. The
Reg-EG and Reg-PEG variants below belong to the same anchored/Tikhonov family
and match the accelerated residual-norm order in comparable settings, while
deriving their anchor placements directly from the operator-side rule and using
a common progress--drift--bias proof.

\section{Operator-side regularization and the Halpern iteration}
\label{sec:operator-side-regularization}

\subsection{The regularize-and-track principle}
\label{sec:framework_principle}

This section and \cref{sec:main-results} share a single design and a common
tracking pattern. The design is to regularize the queried operator with a
vanishing Tikhonov term, apply the standard iterative template to the moving
regularized object, and track the resulting residual. For Picard iteration this
recipe reproduces Halpern iteration; for the forward-step, extragradient,
and past extragradient templates, it produces anchored variants
whose anchor-evaluation pattern is different from common update-level anchored
EG/PEG placements.

For fixed-point problems we allow a general normed space \(\mathcal X\); for
monotone equations we work in a Hilbert space \(\mathcal H\). The regularized
objects are
\begin{equation}\label{eq:reg-objects}
    T_k(\vect{x}) \coloneqq (1 - \delta_k)\,T(\vect{x}) + \delta_k\,\vect{x}_0,
    \qquad
    G_k(\vect{x}) \coloneqq \eta_k\,F(\vect{x}) + \delta_k\,(\vect{x} - \vect{x}_0),
\end{equation}
where \(\vect{x}_0\) is a fixed anchor and \(0<\delta_k\le1\),
\(\delta_k\downarrow0\). The map \(T_k\) is a
\((1-\delta_k)\)-contraction; the map \(G_k\) is the scaled displacement of a
Tikhonov regularization of \(F\), with the operator step size \(\eta_k\) already
folded in. A Picard, forward, EG or PEG step taken on these moving
objects, when expanded, yields an anchored update with the anchor evaluated at
the same point as the corresponding operator query.

The analysis follows a progress--drift--bias pattern. In each setting, an
appropriately chosen residual potential \(E_k\)---the regularized fixed-point
residual for \(T_k\), or a method-specific potential built from \(G_k\)---will
be shown to satisfy a recurrence of the form
\begin{equation}\label{eq:master-recurrence}
    E_{k+1}
    \;\le\;
    (1 - \gamma\,\delta_k)\,\bigl(E_k + C\,(\delta_{k-1} - \delta_k)\,R\bigr),
\end{equation}
for problem-dependent constants \(\gamma,C>0\). The factor
\(1-\gamma\delta_k\) records progress on the current regularized problem; the
additive term records the drift between successive regularized objects. Solving
\eqref{eq:master-recurrence} via the deterministic estimate
\cref{lem:recurrence} bounds \(E_k\), after which a final step bounds the
regularization bias and recovers the original residual. The schedules below are
chosen so that progress, drift, and bias decay at the same rate.

\subsection{The fixed-point prototype: Halpern iteration}

We now run this template on its simplest instance. Let \(\mathcal X\) be a
normed space and let \(T \colon \mathcal X \to \mathcal X\) be nonexpansive with
\(\operatorname{Fix}(T)\ne\varnothing\). Pick a reference
\(\vect{x}^*\in\operatorname{Fix}(T)\). Taking a convex combination of \(T\)
and the constant map \(\vect{x}\mapsto\vect{x}_0\) gives the moving contraction
\begin{equation}
\label{eq:regularized_op}
    T_k(\vect{x})
    \coloneqq (1 - \delta_k)\,T(\vect{x}) + \delta_k\,\vect{x}_0,
\end{equation}
which is a \((1-\delta_k)\)-contraction whenever \(\delta_k\in(0,1]\). A single
Picard step on this map,
\begin{equation}
\label{eq:halpern_iteration}
    \vect{x}_{k+1} = T_k(\vect{x}_k),
\end{equation}
is exactly Halpern iteration.

The connection between Halpern's update and Browder-type contractive
approximants is classical
\cite{halpern1967fixed,browder1967convergence,wittmann1992approximation}.
Quantitative asymptotic-regularity estimates for Halpern iterations have a
longer history: proof-mining arguments gave effective rates in normed spaces
\cite{leustean2007rates,kohlenbach2011quantitative}; Sabach and Shtern obtained
an explicit \(O(1/k)\) residual bound \cite{sabach2017first}; and later work
sharpened constants or studied optimal worst-case bounds in Hilbert and normed
settings \cite{lieder2021convergence,park2022exact,contreras2023optimal}. The
interest here is the explicit conceptual viewpoint in the analysis: \(T_k\) is
contractive, so a Picard step on \(T_k\) shrinks the current regularized
residual; the schedule \(\delta_k=2/(k+2)\) then balances this progress with the
drift from \(T_{k-1}\) to \(T_k\) and the final bias back to \(T\). We give a
short proof of the following known proposition to illustrate the tracking
framework:

\begin{proposition}[Halpern residual bound]\label{prop:halpern_residual_bound}
Let \(T\) be a nonexpansive mapping with a fixed point. Let
\((\vect{x}_k)_{k\ge0}\) be the Halpern iterates
\eqref{eq:halpern_iteration}, initialized at \(\vect{x}_0\), with
\(\delta_k=2/(k+2)\). For any \(\vect{x}^*\in\operatorname{Fix}(T)\),
\[
    r_T(\vect{x}_k)
    =
    \norm{T(\vect{x}_k)-\vect{x}_k}
    \le \frac{8R}{k+1},
    \qquad k\ge1,
\]
where \(R\coloneqq\norm{\vect{x}_0-\vect{x}^*}\).
\end{proposition}

\begin{proof}
We show boundedness, set up a recurrence for the moving residual
\[
    E_k\coloneqq\norm{\vect{x}_k-T_{k-1}(\vect{x}_k)},
\]
solve the recurrence via \cref{lem:recurrence}, and convert back to
\(r_T(\vect{x}_k)\).

\paragraph{Step 1: Boundedness.}
For any \(\delta\in[0,1]\) and any \(\vect{x}\) with
\(\norm{\vect{x}-\vect{x}^*}\le R\), the map
\(\vect{x}\mapsto(1-\delta)T(\vect{x})+\delta\vect{x}_0\) satisfies
\[
    \norm{(1-\delta)T(\vect{x})+\delta\vect{x}_0-\vect{x}^*}
    \le
    (1-\delta)\norm{\vect{x}-\vect{x}^*}
    +\delta\norm{\vect{x}_0-\vect{x}^*}
    \le R,
\]
using \(T(\vect{x}^*)=\vect{x}^*\) and nonexpansiveness of \(T\). Since the
iteration starts at \(\vect{x}_0\), induction gives
\(\norm{\vect{x}_k-\vect{x}^*}\le R\) for all \(k\ge0\), and consequently
\[
    \norm{T(\vect{x}_k)-\vect{x}_0}
    \le
    \norm{T(\vect{x}_k)-T(\vect{x}^*)}
    +\norm{\vect{x}^*-\vect{x}_0}
    \le
    \norm{\vect{x}_k-\vect{x}^*}+R
    \le 2R.
\]

\paragraph{Step 2: The error recurrence.}
Since \(\vect{x}_{k+1}=T_k(\vect{x}_k)\) and \(T_k\) is a
\((1-\delta_k)\)-contraction,
\[
    E_{k+1}
    =
    \norm{T_k(\vect{x}_k)-T_k(T_k(\vect{x}_k))}
    \le
    (1-\delta_k)\norm{\vect{x}_k-T_k(\vect{x}_k)}.
\]
The right-hand factor differs from \(E_k\) by the drift between
\(T_{k-1}\) and \(T_k\). Inserting and removing \(T_{k-1}(\vect{x}_k)\), and
using
\[
    T_{k-1}(\vect{x})-T_k(\vect{x})
    =
    (\delta_{k-1}-\delta_k)(\vect{x}_0-T(\vect{x})),
\]
we obtain
\begin{align*}
    \norm{\vect{x}_k-T_k(\vect{x}_k)}
    &\le
    E_k+\norm{T_{k-1}(\vect{x}_k)-T_k(\vect{x}_k)} \\
    &=
    E_k+(\delta_{k-1}-\delta_k)\norm{\vect{x}_0-T(\vect{x}_k)}
    \le
    E_k+2(\delta_{k-1}-\delta_k)R.
\end{align*}
Thus, for \(k\ge1\),
\begin{equation}\label{eq:error_recurrence}
    E_{k+1}
    \le
    (1-\delta_k)\bigl(E_k+2(\delta_{k-1}-\delta_k)R\bigr).
\end{equation}

\paragraph{Step 3: Solving the recurrence.}
With \(\delta_0=1\), the map \(T_0\) is constant at \(\vect{x}_0\), so
\(\vect{x}_1=\vect{x}_0\) and \(E_1=0\). Applying \cref{lem:recurrence} with
\(\alpha=\beta=2\), \(\gamma=1\), \(c=2\), and \(c_1=2\) gives
\[
    E_k\le\frac{4R}{k+1}.
\]

\paragraph{Step 4: Removing the regularization bias.}
The original residual differs from \(E_k\) by the bias
\[
    T_{k-1}(\vect{x}_k)-T(\vect{x}_k)
    =
    \delta_{k-1}(\vect{x}_0-T(\vect{x}_k)).
\]
Therefore
\[
    \norm{\vect{x}_k-T(\vect{x}_k)}
    \le
    E_k+\delta_{k-1}\norm{\vect{x}_0-T(\vect{x}_k)}
    \le
    \frac{4R}{k+1}+\frac{2}{k+1}\cdot2R
    =
    \frac{8R}{k+1}.
\]
\end{proof}

\paragraph{Extensions.}
The proof is not tied to the exact schedule \(\delta_k=2/(k+2)\). For example,
\(\delta_k=\alpha/(k+\beta)\) with \(1<\alpha\le\beta\) gives the same
\(O(1/k)\) residual decay, with different constants in \cref{lem:recurrence}.
Nor is anchoring outside \(T\) the only fixed-point regularization covered by this
argument. The inner-anchored map
\[
    \widetilde T_k(\vect{x})
    =
    T\bigl((1-\delta_k)\vect{x}+\delta_k\vect{x}_0\bigr)
\]
is again a \((1-\delta_k)\)-contraction, and is the pure inner-anchored
special case of the Tikhonov--Mann scheme \cite{cheval2023modified}. Replacing
the constant anchor by a contraction \(g\) gives the viscosity version
\cite{moudafi2000viscosity,xu2004viscosity}; the residual bound is stated in
\cref{app:viscosity}. Finally, because each \(T_k\) is contractive, one may take
several Picard steps at a fixed regularization level before decreasing
\(\delta_k\). This episodic variant is recorded in
\cref{app:fixed-point-extensions}.

The fixed-point case has now displayed \eqref{eq:master-recurrence} in its
cleanest form. We next replace moving contractions \(T_k\) by moving strongly
monotone operators \(G_k\), and apply the same progress--drift--bias argument to
explicit monotone-equation templates.

\section{Main Results for Monotone Equations}
\label{sec:main-results}
We now instantiate the tracking principle from
\cref{sec:framework_principle} on three monotone-equation templates: the
forward step, extragradient, and past extragradient. The forward-step result is
the minimal stabilization example: the Tikhonov term turns a plain explicit
step, which may be unstable for merely monotone Lipschitz operators, into a
moving contractive problem and yields the \(O(LR/\sqrt{k})\) residual-norm rate
below. The Reg-EG and Reg-PEG results apply the same operator-side rule to the
EG and PEG templates, giving \(O(LR/k)\) residual-norm bounds in the
unconstrained monotone-equation setting. Across all three cases, the object
being tracked is built from the scaled regularized displacement
\[
    G_k(\vect{x})=\eta_k F(\vect{x})+\delta_k(\vect{x}-\vect{x}_0),
\]
with \(\eta_k\equiv\eta\) in the Reg-EG and Reg-PEG sections. The scalar
\(\eta\) or \(\eta_k\) is the step size.

The table gives an overview of the induced methods, parameter scales, and
residual-norm rates. The numerical constants in the theorem statements are
chosen for clean proofs and are not optimized.

\begin{center}
\small
\setlength{\tabcolsep}{2pt}
\renewcommand{\arraystretch}{1.35}
\begin{tabular}{@{}>{\raggedright\arraybackslash}p{0.17\textwidth}>{\raggedright\arraybackslash}p{0.33\textwidth}>{\raggedright\arraybackslash}p{0.28\textwidth}>{\raggedright\arraybackslash}p{0.16\textwidth}@{}}
\toprule
\textbf{Method} & \textbf{Update on \(G_k\)} & \textbf{Parameter scale} & \textbf{Rate for \(\norm{F(\vect{x}_k)}\)} \\
\midrule
Reg-Forward Step & \(\vect{x}_{k+1}=\vect{x}_k-G_k(\vect{x}_k)\) & \(\delta_k\asymp k^{-1}\), \(\eta_k\asymp (L\sqrt{k})^{-1}\) & \(O(LR/\sqrt{k})\) \\
Reg-EG & \(\begin{gathered}\vect{x}_{k+1/2}=\vect{x}_k-G_k(\vect{x}_k)\\ \vect{x}_{k+1}=\vect{x}_k-G_k(\vect{x}_{k+1/2})\end{gathered}\) & \(\delta_k\asymp k^{-1}\), \(\eta\asymp L^{-1}\) & \(O(LR/k)\) \\
Reg-PEG & \(\begin{gathered}\tilde{\vect{x}}_k=\vect{x}_k-G_k(\tilde{\vect{x}}_{k-1})\\ \vect{x}_{k+1}=\vect{x}_k-G_k(\tilde{\vect{x}}_k)\end{gathered}\) & \(\delta_k\asymp k^{-1}\), \(\eta\asymp L^{-1}\) & \(O(LR/k)\) \\
\bottomrule
\end{tabular}
\end{center}

\subsection{The Regularized Monotone Equation}
\label{sec:reg-monotone-eqn}

For the three monotone-equation theorem statements below, unless stated
otherwise, \(F:\mathcal H\to\mathcal H\) is a single-valued full-domain
monotone \(L\)-Lipschitz operator on a real Hilbert space with \(L>0\), and
\(\operatorname{Zer}(F)\ne\varnothing\). We fix an anchor \(\vect{x}_0\), choose
a solution \(\vect{x}^*\in\operatorname{Zer}(F)\), and write
\[
    R\coloneqq\norm{\vect{x}_0-\vect{x}^*}.
\]
For \(\delta>0\), define
\begin{equation}\label{eq:reg-eqn}
    \mathcal F_\delta(\vect{x})
    \coloneqq F(\vect{x})+\delta(\vect{x}-\vect{x}_0).
\end{equation}
We refer to \(\mathcal F_\delta(\vect{x})=0\) as the regularized monotone
equation and denote its solution, when it exists, by \(\vect{x}_\delta^*\). The
facts below record existence, the basic location bound, and stability as
\(\delta\) varies. The EG and PEG trajectory estimates use the moving
regularized zeros; the forward-step proof uses the same regularized residual
viewpoint but compares directly with the original zero \(\vect{x}^*\).

The first point is that the regularized equation has a unique single zero at
each regularization level.

\begin{lemma}[Existence of the regularized zero]\label{lem:reg-zero-existence}
For every \(\delta>0\), the regularized monotone equation
\(\mathcal F_\delta(\vect{x})=0\) has a unique solution.
\end{lemma}
\begin{proof}
Since \(F\) is everywhere defined and Lipschitz, it is continuous. The
maximal-monotonicity theorem for full-domain continuous monotone operators on
Hilbert spaces gives that \(F\) is maximally monotone
\cite[Cor.~20.28]{bauschke2017convex}. Therefore \(F+\delta I\) is maximally monotone and
\(\delta\)-strongly monotone. By Browder--Minty surjectivity for maximally
monotone strongly monotone operators \cite[Prop.~22.11(ii)]{bauschke2017convex},
\(\operatorname{ran}(F+\delta I)=\mathcal H\). Hence there exists
\(\vect{x}\in\mathcal H\) such that
\[
    (F+\delta I)(\vect{x})=\delta\vect{x}_0,
\]
which is equivalent to \(\mathcal F_\delta(\vect{x})=0\). Strong monotonicity
gives uniqueness.
\end{proof}

The next estimate says that this target stays in the ball determined by the
anchor and any fixed solution of the original equation. This is the elementary
geometric fact used later to remove the regularization bias.

\begin{lemma}[Boundedness of the regularized zero]\label{lem:reg-and-x0}
For every \(\delta>0\),
\[
    \norm{\vect{x}_\delta^*-\vect{x}^*}\le R,
    \qquad
    \norm{\vect{x}_\delta^*-\vect{x}_0}\le R.
\]
\end{lemma}
\begin{proof}
Since \(\mathcal F_\delta(\vect{x}_\delta^*)=0\), we have
\(F(\vect{x}_\delta^*)=-\delta(\vect{x}_\delta^*-\vect{x}_0)\). Monotonicity of
\(F\) and \(F(\vect{x}^*)=\vect{0}\) give
\[
    0
    \le
    \angles{F(\vect{x}_\delta^*)-F(\vect{x}^*),
            \vect{x}_\delta^*-\vect{x}^*}
    =
    -\delta
    \angles{\vect{x}_\delta^*-\vect{x}_0,
            \vect{x}_\delta^*-\vect{x}^*}.
\]
Thus
\(\angles{\vect{x}_\delta^*-\vect{x}^*,
          \vect{x}_\delta^*-\vect{x}_0}\le0\). Plugging
\[
    2\angles{\vect{a},\vect{b}}
    =
    \norm{\vect{a}}^2+\norm{\vect{b}}^2-\norm{\vect{a}-\vect{b}}^2
\]
with \(\vect{a}=\vect{x}_\delta^*-\vect{x}^*\) and
\(\vect{b}=\vect{x}_\delta^*-\vect{x}_0\), we obtain the stronger bound
\[
    \norm{\vect{x}_\delta^*-\vect{x}^*}^2
    +\norm{\vect{x}_\delta^*-\vect{x}_0}^2
    \le
    \norm{\vect{x}^*-\vect{x}_0}^2
    =
    R^2.
\]
Hence, both displayed bounds follow.
\end{proof}

Finally, we show that the regularized target changes slowly as the
regularization level varies. This is the form needed in the trajectory bounds for the EG and
PEG variants.

\begin{lemma}[Stability of the regularized zero]\label{lem:stability-reg}
For \(\delta_1,\delta_2>0\),
\[
    \norm{\vect{x}_{\delta_1}^*-\vect{x}_{\delta_2}^*}
    \le
    \frac{|\delta_1-\delta_2|}{\max\{\delta_1,\delta_2\}}\,R.
\]
\end{lemma}
\begin{proof}
Set \(\vect{u}=\vect{x}_{\delta_1}^*\),
\(\vect{v}=\vect{x}_{\delta_2}^*\), and \(\vect{d}=\vect{u}-\vect{v}\). If
\(\vect{d}=\vect{0}\), the claim is immediate. Otherwise, using
\[
    F(\vect{u})=-\delta_1(\vect{u}-\vect{x}_0),
    \qquad
    F(\vect{v})=-\delta_2(\vect{v}-\vect{x}_0),
\]
and monotonicity of \(F\),
\[
\begin{aligned}
0
&\le
\angles{F(\vect{u})-F(\vect{v}),\vect{d}} \\
&=
-\delta_1\norm{\vect{d}}^2
+(\delta_2-\delta_1)\angles{\vect{v}-\vect{x}_0,\vect{d}}.
\end{aligned}
\]
Cauchy--Schwarz and \cref{lem:reg-and-x0} then give
\[
    \delta_1\norm{\vect{d}}
    \le
    |\delta_1-\delta_2|\,\norm{\vect{v}-\vect{x}_0}
    \le
    |\delta_1-\delta_2|\,R.
\]
Repeating the same argument with the roles of \(\delta_1\) and \(\delta_2\)
swapped gives the same bound with \(\delta_2\) on the left. Taking the tighter
of the two inequalities yields the denominator \(\max\{\delta_1,\delta_2\}\).
\end{proof}

The methods below apply explicit templates to the scaled displacement
\[
    G_k(\vect{x})
    =
    \eta_kF(\vect{x})+\delta_k(\vect{x}-\vect{x}_0).
\]
Since \(\eta_kF\) is again monotone and has the same zero set as \(F\), the
preceding lemmas apply directly to \(G_k\) as the regularization of the operator
\(\eta_kF\) with parameter \(\delta_k\). In the Reg-EG and Reg-PEG sections,
where \(\eta_k\equiv\eta\), we denote the zero of \(G_k\) by \(\vect{x}_k^*\).
For the non-increasing schedules used there, \cref{lem:stability-reg} gives
\[
    \norm{\vect{x}_{k-1}^*-\vect{x}_k^*}
    \le
    \frac{\delta_{k-1}-\delta_k}{\delta_{k-1}}\,R,
\]
which is the moving-zero estimate used in the trajectory bounds. The same
operator \(G_k\) is also \(\delta_k\)-strongly monotone and has Lipschitz
constant at most \(\eta_kL+\delta_k\), which are the constants used in the
one-step progress estimates.


\subsection{The Tikhonov-Regularized Forward Step}
\label{subsec:reg-forward}

In this section, we discuss the forward step as the simplest
monotone equation template. It queries \(F\) once per iteration and updates by
the explicit rule
\[
    \vect{x}_{k+1} = \vect{x}_k - \eta_k F(\vect{x}_k),
    \qquad \eta_k > 0.
\]
When \(F=\nabla f\), this is gradient descent; when \(F\) is the saddle
operator of a minmax problem, it is gradient descent-ascent. If \(F\) is
\(1/L\)-cocoercive, then \(I-\eta F\) is nonexpansive for
\(\eta\in(0,2/L]\), and the usual averaged-map theory gives convergence under
the standard stricter step-size conditions. Under mere monotonicity and
Lipschitz continuity, however, it is well known that the forward step can
diverge. In finite dimensions, the standard example is a full-rank
skew-symmetric linear map (which includes bilinear minmax optimization)
\(F(\vect{x})=\matr{A}\vect{x}\), for which
\[
    \norm{\vect{x}-\eta F(\vect{x})}^2
    =
    \norm{\vect{x}}^2+\eta^2\norm{\matr{A}\vect{x}}^2,
\]
so every nonzero step strictly recedes from the origin. Extragradient repairs
this failure by using a second query, while PEG/Popov uses a staggered past
query. Interestingly, the results in this section show that the iteratively
regularized forward step alone achieves stabilization at the same last-iterate
residual-norm rate as these two-query methods.

Recall from \cref{sec:framework_principle} the scaled regularized displacement
\[
    G_k(\vect{x})
    \coloneqq
    \eta_k F(\vect{x})+\delta_k(\vect{x}-\vect{x}_0),
\]
where the operator step size \(\eta_k\) has been folded into \(G_k\). Running a
forward step on \(G_k\) produces the Tikhonov-regularized forward step
\begin{equation}\label{eq:reg_forward_step}
    \vect{x}_{k+1}
    =
    \vect{x}_k-G_k(\vect{x}_k)
    =
    \vect{x}_k-\eta_kF(\vect{x}_k)-\delta_k(\vect{x}_k-\vect{x}_0).
\end{equation}
Equivalently, this is the anchored forward step with operator weight
\(\eta_k\) and anchor weight \(\delta_k\), with both terms evaluated at
\(\vect{x}_k\). The update calls \(F\) once per iteration. With the schedule
\(2\eta_kL=\sqrt{\delta_k}\) used below, the map
\(\vect{x}\mapsto\vect{x}-G_k(\vect{x})\) is a contraction at each fixed
\(k\); the Tikhonov term supplies the stability that the unregularized forward
step lacks for merely monotone Lipschitz operators.


\begin{theorem}[Tikhonov-Regularized Forward Residual Bound]\label{thm:reg_forward}
Let \(F:\mathcal H\to\mathcal H\) be a monotone, \(L\)-Lipschitz operator with
\(L>0\) and with a zero \(\vect{x}^*\in\mathcal H\). Let
\(R\coloneqq\norm{\vect{x}_0-\vect{x}^*}\). Consider the iteration
\eqref{eq:reg_forward_step} with parameters satisfying
\(2\eta_kL=\sqrt{\delta_k}\) and \(\delta_k=4/(k+4)\) for \(k\ge0\). Then, for
every \(k\ge1\),
\[
    r_F(\vect{x}_k)
    =
    \norm{F(\vect{x}_k)}
    \le
    \frac{24LR}{\sqrt{k+3}}.
\]
\end{theorem}

We record a few remarks on this result.
The rate \(O(LR/\sqrt{k})\) has the same
last-iterate residual rate as unanchored extragradient and Popov/PEG
\cite{gorbunov2022extragradient,gorbunov2022last}, attained here with the
literal one-call forward template; Reg-EG and Reg-PEG below improve it to
\(O(LR/k)\). If \(F\) is also \(1/L\)-cocoercive, the map
\(\vect{x}\mapsto\vect{x}-\eta F(\vect{x})\) is then nonexpansive and Halpern
iteration already gives \(O(LR/k)\).
In contrast to the constant \(\eta=1/(8L)\) used by
the EG and PEG variants below, we require a vanishing step size
without the extragradient or past-query stabilization.

\paragraph{Proof outline.}
The argument follows the progress-drift-bias pattern of
\cref{sec:framework_principle}, applied to the regularized residual
\[
    E_k\coloneqq\norm{G_{k-1}(\vect{x}_k)},\qquad k\ge1.
\]
Step 1 shows that \(\vect{x}\mapsto\vect{x}-G_k(\vect{x})\) is a
\((1-\delta_k/2)\)-contraction. Step 2 uses this contraction and the identity
\(G_k(\vect{x}^*)=\delta_k(\vect{x}^*-\vect{x}_0)\) to keep the trajectory in a
ball of radius \(2R\) around \(\vect{x}^*\). Step 3 closes the recurrence on
\(E_k\): it inherits the contraction factor from Step 1 and absorbs a drift
term \(3(\delta_{k-1}-\delta_k)R\) produced by the change
\(G_{k-1}\to G_k\). Finally, \cref{lem:recurrence} bounds \(E_k\), and the
identity \(G_{k-1}=\eta_{k-1}F+\delta_{k-1}(\cdot-\vect{x}_0)\) converts this
back to a bound on \(\norm{F(\vect{x}_k)}\).

\begin{proof}[Proof of \cref{thm:reg_forward}]
We carry out the three steps in turn.

\paragraph{Step 1: The regularized update is a contraction.}
Consider the map \(S_k(\vect{x})\coloneqq\vect{x}-G_k(\vect{x})\). For any
\(\vect{x},\vect{y}\in\mathcal H\),
\[
    S_k(\vect{x})-S_k(\vect{y})
    =
    (1-\delta_k)(\vect{x}-\vect{y})
    -\eta_k\bigl(F(\vect{x})-F(\vect{y})\bigr).
\]
Squaring and using monotonicity of \(F\) to drop the cross term gives
\begin{align*}
    \norm{S_k(\vect{x})-S_k(\vect{y})}^2
    &\le
    (1-\delta_k)^2\norm{\vect{x}-\vect{y}}^2
    +\eta_k^2\norm{F(\vect{x})-F(\vect{y})}^2 \\
    &\le
    \bigl((1-\delta_k)^2+\eta_k^2L^2\bigr)
    \norm{\vect{x}-\vect{y}}^2.
\end{align*}
Substituting \(\eta_k^2L^2=\delta_k/4\) and using
\(\delta_k\in(0,1]\),
\[
    (1-\delta_k)^2+\frac{\delta_k}{4}
    =
    1-2\delta_k+\delta_k^2+\frac{\delta_k}{4}
    \le
    \left(1-\frac{\delta_k}{2}\right)^2.
\]
Thus \(S_k\) is a \((1-\delta_k/2)\)-contraction.

\paragraph{Step 2: The trajectory is bounded.}
Since \(F(\vect{x}^*)=\vect{0}\),
\[
    G_k(\vect{x}^*)=\delta_k(\vect{x}^*-\vect{x}_0).
\]
Therefore,
\[
    \norm{\vect{x}_{k+1}-\bigl(\vect{x}^*-G_k(\vect{x}^*)\bigr)}
    =
    \norm{S_k(\vect{x}_k)-S_k(\vect{x}^*)}
    \le
    \left(1-\frac{\delta_k}{2}\right)\norm{\vect{x}_k-\vect{x}^*}.
\]
The triangle inequality gives
\[
    \norm{\vect{x}_{k+1}-\vect{x}^*}
    \le
    \left(1-\frac{\delta_k}{2}\right)\norm{\vect{x}_k-\vect{x}^*}
    +\delta_k R.
\]
Induction yields \(\norm{\vect{x}_k-\vect{x}^*}\le2R\) for all \(k\ge0\): the
base case is immediate, and if \(\norm{\vect{x}_k-\vect{x}^*}\le2R\), then
\[
    \norm{\vect{x}_{k+1}-\vect{x}^*}
    \le
    \left(1-\frac{\delta_k}{2}\right)2R+\delta_kR
    =
    2R.
\]
Consequently,
\[
    \norm{\vect{x}_k-\vect{x}_0}
    \le
    \norm{\vect{x}_k-\vect{x}^*}+\norm{\vect{x}^*-\vect{x}_0}
    \le
    3R.
\]

\paragraph{Step 3: Recurrence and final bound.}
Using \(\vect{x}_{k+1}=\vect{x}_k-G_k(\vect{x}_k)\) and the contraction of
\(S_k\),
\begin{align*}
    \norm{G_k(\vect{x}_{k+1})}
    &=
    \norm{S_k(\vect{x}_{k+1})-S_k(\vect{x}_k)}
    \qquad(\text{since }\vect{x}_{k+1}=S_k(\vect{x}_k)) \\
    &\le
    \left(1-\frac{\delta_k}{2}\right)\norm{\vect{x}_{k+1}-\vect{x}_k}
    =
    \left(1-\frac{\delta_k}{2}\right)\norm{G_k(\vect{x}_k)}.
\end{align*}
We next compare \(G_k(\vect{x}_k)\) to
\(E_k=\norm{G_{k-1}(\vect{x}_k)}\). From the definition of \(G_k\),
\[
    G_k(\vect{x}_k)
    =
    \frac{\eta_k}{\eta_{k-1}}G_{k-1}(\vect{x}_k)
    +
    \left(\delta_k-\frac{\eta_k}{\eta_{k-1}}\delta_{k-1}\right)
    (\vect{x}_k-\vect{x}_0).
\]
Since \(\delta_k\le\delta_{k-1}\) and
\(\eta_k/\eta_{k-1}=\sqrt{\delta_k/\delta_{k-1}}\le1\),
\[
    \left|\delta_k-\frac{\eta_k}{\eta_{k-1}}\delta_{k-1}\right|
    =
    \sqrt{\delta_k\delta_{k-1}}-\delta_k
    \le
    \delta_{k-1}-\delta_k.
\]
Together with \(\norm{\vect{x}_k-\vect{x}_0}\le3R\), this gives
\[
    \norm{G_k(\vect{x}_k)}
    \le
    \frac{\eta_k}{\eta_{k-1}}E_k
    +(\delta_{k-1}-\delta_k)\norm{\vect{x}_k-\vect{x}_0}
    \le
    E_k+3(\delta_{k-1}-\delta_k)R.
\]
Setting \(E_{k+1}=\norm{G_k(\vect{x}_{k+1})}\), we obtain, for \(k\ge1\),
\[
    E_{k+1}
    \le
    \left(1-\frac{\delta_k}{2}\right)
    \bigl(E_k+3(\delta_{k-1}-\delta_k)R\bigr).
\]
The base case follows from the contraction of \(S_0\), the identity
\(G_0(\vect{x}_0)=\eta_0F(\vect{x}_0)\), and Lipschitzness around
\(\vect{x}^*\):
\[
    E_1
    =
    \norm{G_0(\vect{x}_1)}
    \le
    \left(1-\frac{\delta_0}{2}\right)\norm{\eta_0F(\vect{x}_0)}
    \le
    \frac{R}{4}.
\]
Here \(\delta_0=1\), \(\eta_0=1/(2L)\), and
\(\norm{F(\vect{x}_0)}=\norm{F(\vect{x}_0)-F(\vect{x}^*)}\le LR\).
Applying \cref{lem:recurrence} with
\[
    \alpha=\beta=4,\qquad
    \gamma=\frac12,\qquad
    c=3,\qquad
    c_1=\frac14
\]
yields
\[
    E_k\le\frac{12R}{k+3}.
\]
Finally, since
\[
    G_{k-1}(\vect{x}_k)
    =
    \eta_{k-1}F(\vect{x}_k)
    +\delta_{k-1}(\vect{x}_k-\vect{x}_0),
\]
we have
\[
    \norm{F(\vect{x}_k)}
    \le
    \frac{E_k+3\delta_{k-1}R}{\eta_{k-1}}
    \le
    \frac{24LR}{\sqrt{k+3}},
\]
where the last inequality uses
\(\delta_{k-1}=4/(k+3)\), \(\eta_{k-1}=1/(L\sqrt{k+3})\), and
\(E_k\le12R/(k+3)\).
\end{proof}


\subsection{A Regularized Extragradient Method}
\label{subsec:reg-eg}

For a monotone Lipschitz operator \(F\) that need not be cocoercive,
extragradient provides stabilization and achieves last-iterate residual-norm
guarantees of order \(O(LR/\sqrt{k})\)
\cite{gorbunov2022extragradient}. Anchored EG-family variants, including EAG,
FEG, and the flexible anchored G-EAG family, achieve the accelerated
\(O(LR/k)\) residual-norm rate in their respective settings
\cite{yoon2021accelerated,lee2021fast,boct2024extra}. In this subsection,
we show that the operator-side substitution of \cref{sec:framework_principle}
leads to a different anchor placement and the same last-iterate
residual-norm rate \(O(LR/k)\).

Applying EG to the scaled regularized displacement
\[
    G_k(\vect{x}) \coloneqq \eta F(\vect{x}) + \delta_k(\vect{x} - \vect{x}_0)
\]
gives the regularized extragradient (Reg-EG) method:
\begin{equation} \label{eq:reg_eg}
    \vect{x}_{k+1/2} = \vect{x}_k - G_k(\vect{x}_k),
    \qquad
    \vect{x}_{k+1}   = \vect{x}_k - G_k(\vect{x}_{k+1/2}),
\end{equation}
or, equivalently, after expanding \(G_k\),
\[
    \vect{x}_{k+1/2}
    = \vect{x}_k-\eta F(\vect{x}_k)-\delta_k(\vect{x}_k-\vect{x}_0),
    \qquad
    \vect{x}_{k+1}
    = \vect{x}_k-\eta F(\vect{x}_{k+1/2})
      -\delta_k(\vect{x}_{k+1/2}-\vect{x}_0).
\]

\paragraph{Comparison with anchored EG variants.}
The relevant comparator is the base-anchored EG pattern introduced by
Yoon--Ryu's EAG method \cite{yoon2021accelerated}. With fixed anchor
\(\vect{x}_0\), the EAG and flexible G-EAG updates share, up to coefficient
choices and affine rescalings, the skeleton \cite{boct2024extra}
\[
    \vect{x}_{k+1/2}
    =
    \vect{x}_k+a_k(\vect{x}_0-\vect{x}_k)-b_kF(\vect{x}_k),
    \qquad
    \vect{x}_{k+1}
    =
    \vect{x}_k+\bar a_k(\vect{x}_0-\vect{x}_k)-\bar b_kF(\vect{x}_{k+1/2}).
\]
In this pattern, both anchor
corrections are tied to the base displacement \(\vect{x}_0-\vect{x}_k\).
Reg-EG has the same first-line form but differs in the second line:
\[
    \vect{x}_{k+1}
    =
    \vect{x}_k-\eta F(\vect{x}_{k+1/2})
    -\delta_k(\vect{x}_{k+1/2}-\vect{x}_0),
\]
where the anchor is placed at the lookahead point itself.

The main Reg-EG bound is as follows.
\begin{theorem}[Last-Iterate Residual Bound for Reg-EG] \label{thm:reg_eg}
Let \(F:\mathcal H\to\mathcal H\) be a monotone, \(L\)-Lipschitz operator with \(L>0\) and a solution \(\vect{x}^*\) such that \(F(\vect{x}^*) = \vect{0}\). Consider the regularized extragradient method \eqref{eq:reg_eg} with parameters \(\eta = \frac{1}{8L}\) and \(\delta_k = \frac{4}{k+32}\) for \(k \ge 0\). Let \(R \coloneqq \norm{\vect{x}_0 - \vect{x}^*}\). Then, for all \(k\ge1\), the iterates satisfy:
\[
    r_F(\vect{x}_k)=\norm{F(\vect{x}_k)} \le \frac{128LR}{k+31}.
\]
\end{theorem}

\paragraph{Proof outline.}
The proof follows the shared progress--drift--bias architecture of
\cref{sec:framework_principle}, with one extra ingredient relative to the
regularized forward-step argument. The Reg-EG update map
\(\vect{x}\mapsto \vect{x}-G_k(\vect{x}-G_k(\vect{x}))\) is not a contraction,
so the one-step contraction argument used for the forward step does not apply
directly. We instead use two fixed-\(G_k\) EG estimates. The first,
\cref{lem:eg-linear-conv}, contracts the distance to the zero of a fixed
strongly monotone Lipschitz operator; together with stability of the moving
regularized zeros, this gives the trajectory bound
\(\norm{\vect{x}_k-\vect{x}_0}\le2R\) in \cref{lem:traj-bound}. The second,
proved below, is the residual contraction
\[
    \norm{G_k(\vect{x}_{k+1})}
    \le
    \left(1-\frac{\delta_k}{2}\right)\norm{G_k(\vect{x}_k)}.
\]
The drift from \(G_{k-1}\) to \(G_k\) then yields a recurrence for
\(r_{k-1}(\vect{x}_k)=\norm{G_{k-1}(\vect{x}_k)}\), controlled by the trajectory
bound and solved by \cref{lem:recurrence}. Finally,
\(G_{k-1}=\eta F+\delta_{k-1}(\cdot-\vect{x}_0)\) removes the regularization
bias and recovers the stated bound on \(r_F(\vect{x}_k)\).

\subsubsection{Preparatory Lemmas}

Linear convergence of extragradient in strongly monotone Lipschitz variational inequalities is known
\cite{tseng1995linear}. We include the following lemma for completeness with the constants needed below.

\begin{lemma}[Linear Convergence of Extragradient]\label{lem:eg-linear-conv}
Let \(G: \mathcal{H} \to \mathcal{H}\) be \(\bar L\)-Lipschitz with \(\bar L>0\) and \(\mu\)-strongly monotone with \(\mu>0\). Let \(G(\vect{x}_G^*)=0\), and let
\[
    \vect{x}_{t+1/2}=\vect{x}_t-\alpha G(\vect{x}_t),
    \qquad
    \vect{x}_{t+1}=\vect{x}_t-\alpha G(\vect{x}_{t+1/2})
\]
with \(0<\alpha\le 1/(4\bar L)\). Then
\[
    \norm{\vect{x}_{t+1} - \vect{x}_G^*}
    \le
    \left(1-\frac{\alpha\mu}{2}\right)
    \norm{\vect{x}_t - \vect{x}_G^*}.
\]
\end{lemma}

\begin{proof}
We first establish a lower bound for $\norm{\vect{x}_{t+1/2} - \vect{x}_G^*}^2$:
\begin{align*}
\norm{\vect{x}_{t+1/2} - \vect{x}_G^*}^2
&= \norm{\vect{x}_t - \vect{x}_G^* - \alpha (G(\vect{x}_t) - G(\vect{x}_G^*))}^2 \\
&\ge \norm{\vect{x}_t - \vect{x}_G^*}^2 - 2\alpha\bar L \norm{\vect{x}_t - \vect{x}_G^*}^2 \\
&= \parenth{1 - 2\alpha\bar L}\norm{\vect{x}_t - \vect{x}_G^*}^2.
\end{align*}
The inequality uses Cauchy-Schwarz and the \(\bar L\)-Lipschitz continuity of \(G\).

Next, we bound $\norm{\vect{x}_{t+1} - \vect{x}_G^*}^2$:
\begin{align}
\norm{\vect{x}_{t+1} - \vect{x}_G^*}^2
&= \norm{\vect{x}_t - \vect{x}_G^* - \alpha G(\vect{x}_{t+1/2})}^2 \nonumber \\
&= \norm{\vect{x}_t - \vect{x}_G^*}^2
 - 2\alpha \langle \vect{x}_t - \vect{x}_G^*, G(\vect{x}_{t+1/2}) \rangle
 + \alpha^2 \norm{G(\vect{x}_{t+1/2})}^2. \label{eq:main_expansion}
\end{align}
We decompose the inner product term as follows:
\[
\langle \vect{x}_t - \vect{x}_G^*, G(\vect{x}_{t+1/2}) \rangle
= \langle \vect{x}_{t+1/2} - \vect{x}_G^*, G(\vect{x}_{t+1/2}) \rangle
 + \alpha \langle G(\vect{x}_t), G(\vect{x}_{t+1/2}) \rangle.
\]
By \(\mu\)-strong monotonicity, the first term on the right-hand side is bounded below by \(\mu \norm{\vect{x}_{t+1/2} - \vect{x}_G^*}^2\). For the second term, we use the identity \(2\langle \vect{a}, \vect{b} \rangle = \norm{\vect{a}}^2 + \norm{\vect{b}}^2 - \norm{\vect{a}-\vect{b}}^2\) and the \(\bar L\)-Lipschitz property:
\begin{align*}
2\alpha^2 \langle G(\vect{x}_t), G(\vect{x}_{t+1/2}) \rangle
&\geq \alpha^2 \parenth{\norm{G(\vect{x}_t)}^2 + \norm{G(\vect{x}_{t+1/2})}^2 - \bar L^2 \norm{\vect{x}_t - \vect{x}_{t+1/2}}^2} \\
&= \alpha^2(1 - \alpha^2 \bar L^2)\norm{G(\vect{x}_t)}^2 + \alpha^2\norm{G(\vect{x}_{t+1/2})}^2.
\end{align*}
Substituting these bounds into \eqref{eq:main_expansion}:
\begin{align*}
\norm{\vect{x}_{t+1} - \vect{x}_G^*}^2
&\leq \norm{\vect{x}_t - \vect{x}_G^*}^2
 - 2\alpha\mu \norm{\vect{x}_{t+1/2} - \vect{x}_G^*}^2
 - \alpha^2(1 - \alpha^2\bar L^2)\norm{G(\vect{x}_t)}^2.
\end{align*}
Dropping the final non-positive term and using \(1-2\alpha\bar L\ge1/2\), we get
\[
\norm{\vect{x}_{t+1} - \vect{x}_G^*}^2
\le
\parenth{1-\alpha\mu}\norm{\vect{x}_t - \vect{x}_G^*}^2.
\]
Since \(\mu\le\bar L\) and \(\alpha\le1/(4\bar L)\), we have \(\alpha\mu\in[0,1]\). Taking square roots and using \(\sqrt{1-z}\le1-z/2\) with \(z=\alpha\mu\) proves the claim.
\end{proof}

\begin{lemma}[Boundedness of the Trajectory]\label{lem:traj-bound}
Let \(F:\mathcal H\to\mathcal H\) be monotone and \(L\)-Lipschitz with \(L>0\), and let \(\vect{x}^*\) satisfy \(F(\vect{x}^*)=\vect{0}\). Set \(R\coloneqq\norm{\vect{x}_0-\vect{x}^*}\). Let \((\vect{x}_k)_{k\ge0}\) be generated by the Reg-EG scheme \eqref{eq:reg_eg} with parameters \(\eta=\frac{1}{8L}\) and a non-increasing sequence \((\delta_k)\) satisfying \(\delta_0 \le \frac{1}{8}\). Denote by \(\vect{x}_k^*\) the unique zero of \(G_k\). If \(\delta_k = \frac{c}{k+8c}\) for some \(c \ge 2\), then for every \(k\ge0\), we have \(\norm{\vect{x}_k - \vect{x}_0} \le 2R\).
\end{lemma}
\begin{proof}
At phase \(k\), the algorithm performs one EG step on the operator \(G_k\). This operator is \(\delta_k\)-strongly monotone and has Lipschitz constant at most \(\eta L+\delta_k\). With our choice of \(\eta = 1/(8L)\) and \(\delta_k \le 1/8\), we have \(\eta L+\delta_k \le 1/4\). We can apply Lemma~\ref{lem:eg-linear-conv} with \(G=G_k\), \(\mu=\delta_k\), \(\bar L=1/4\), and step size \(\alpha=1\). This gives the local contraction around the zero \(\vect{x}_k^*\):
\[
    \norm{\vect{x}_{k+1} - \vect{x}_k^*}
    \le
    \left(1-\frac{\delta_k}{2}\right)
    \norm{\vect{x}_k - \vect{x}_k^*}.
\]
For \(k\ge1\), the triangle inequality and Lemma~\ref{lem:stability-reg},
applied to the monotone operator \(\eta F\), give the following bound, which
handles the shift from \(\vect{x}_k^*\) to \(\vect{x}_{k-1}^*\):
\begin{align*}
    \norm{\vect{x}_k - \vect{x}_k^*} &\le \norm{\vect{x}_k - \vect{x}_{k-1}^*} + \norm{\vect{x}_{k-1}^* - \vect{x}_k^*} \\
    &\le \norm{\vect{x}_k - \vect{x}_{k-1}^*} + \frac{\delta_{k-1}-\delta_k}{\delta_{k-1}} \norm{\vect{x}^*-\vect{x}_0}.
\end{align*}
Combining these yields the recurrence:
\[
    \norm{\vect{x}_{k+1} - \vect{x}_k^*} \le \left(1 - \frac{\delta_k}{2}\right) \left(\norm{\vect{x}_k - \vect{x}_{k-1}^*} + \frac{\delta_{k-1} - \delta_k}{\delta_{k-1}} R \right),\qquad k\ge1.
\]
We show \(\norm{\vect{x}_k - \vect{x}_{k-1}^*} \le R\) by induction for \(k\ge1\). For the base case \(k=1\), the local contraction gives \(\norm{\vect{x}_1-\vect{x}_0^*}\le\norm{\vect{x}_0-\vect{x}_0^*}\), and Lemma~\ref{lem:reg-and-x0} gives \(\norm{\vect{x}_0-\vect{x}_0^*}\le R\). For the inductive step, assume \(\norm{\vect{x}_k - \vect{x}_{k-1}^*} \le R\). With \(\delta_k = \frac{c}{k+8c}\), we have
\[
\frac{\delta_{k-1}-\delta_k}{\delta_{k-1}}=\frac{1}{k+8c},
\qquad
\frac{\delta_k}{2}=\frac{c}{2(k+8c)}\ge\frac{1}{k+8c},
\]
where the last inequality uses \(c\ge2\). The recurrence therefore gives \(\norm{\vect{x}_{k+1} - \vect{x}_k^*} \le R\).
Indeed, with \(a=\delta_k/2\) and \(q=(\delta_{k-1}-\delta_k)/\delta_{k-1}\), the displayed inequalities give \(q\le a\), and hence \((1-a)(1+q)\le(1-a)(1+a)\le1\).
The case \(k=0\) in the lemma is immediate, since \(\norm{\vect{x}_0-\vect{x}_0}=0\). For \(k\ge0\), the established induction bound at index \(k+1\), Lemma~\ref{lem:reg-and-x0} applied through the scaled-operator reduction stated after \cref{lem:stability-reg}, and the triangle inequality give
\[
    \norm{\vect{x}_{k+1} - \vect{x}_0} \le \norm{\vect{x}_{k+1} - \vect{x}_k^*} + \norm{\vect{x}_k^* - \vect{x}_0} \le R + R = 2R.
\]
\end{proof}

\subsubsection{Proof of Theorem~\ref{thm:reg_eg}}
\noindent\textbf{Step 1: Fixed-operator residual contraction.}
For a fixed phase \(k\), define the following shorthand notation for the residuals at step $k$:
\[
  \vect{v}_1 = G_{k}(\vect{x}_k),\quad
  \vect{v}_2 = G_{k}(\vect{x}_{k+1/2}),\quad
  \vect{v}_3 = G_{k}(\vect{x}_{k+1}).
\]
A direct expansion of the difference of squared norms yields:
\begin{align*}
  \|\vect{v}_3\|^2 - \|\vect{v}_1\|^2
  &= \bigl\|(\vect{v}_3 - \vect{v}_2) + \vect{v}_2\bigr\|^2
   - \bigl\|(\vect{v}_1 - \vect{v}_2) + \vect{v}_2\bigr\|^2 \\[-0.3em]
  &= \|\vect{v}_3 - \vect{v}_2\|^2 \;-\;\|\vect{v}_1 - \vect{v}_2\|^2
   \;+\;2\bigl\langle \vect{v}_3 - \vect{v}_1,\;\vect{v}_2\bigr\rangle \\[-0.3em]
  &\le \bigl((\eta L + \delta_{k})^2 - 1\bigr)\,\|\vect{v}_1 - \vect{v}_2\|^2
   \;-\;2\,\delta_{k}\,\|\vect{v}_2\|^2,
\end{align*}
where the final inequality uses Lipschitz continuity and strong monotonicity of
$G_k$. For Lipschitz continuity,
\[
\|\vect{v}_3 - \vect{v}_2\|\le(\eta L + \delta_{k})\|\vect{x}_{k+1}-\vect{x}_{k+1/2}\|
=(\eta L + \delta_{k})\,\|\vect{v}_1 - \vect{v}_2\|.
\]
For strong monotonicity, since \(\vect{x}_{k+1}-\vect{x}_k=-\vect{v}_2\),
\[
\angles{\vect{v}_3-\vect{v}_1,-\vect{v}_2}
=
\angles{G_k(\vect{x}_{k+1})-G_k(\vect{x}_k),\vect{x}_{k+1}-\vect{x}_k}
\ge
\delta_k\norm{\vect{x}_{k+1}-\vect{x}_k}^2
=
\delta_k\norm{\vect{v}_2}^2.
\]
Thus \(\langle \vect{v}_3-\vect{v}_1,\vect{v}_2\rangle\le -\delta_k\norm{\vect{v}_2}^2\).

We can further simplify this expression. Given our choice of parameters, we have:
\[
1 - (\eta L + \delta_{k})\;\ge\; 1 - (\eta L + \delta_0) \;=\; \tfrac34.
\]
Furthermore, applying the Lipschitz continuity of $G_k$ between $\vect{x}_k$ and $\vect{x}_{k+1/2}$ gives:
\[
\|\vect{v}_1 - \vect{v}_2\|\le(\eta L + \delta_{k})\|\vect{v}_1\|,
\]
and then, by the reverse-triangle inequality,
\[
\|\vect{v}_2\|\ge\|\vect{v}_1\| - \|\vect{v}_1 - \vect{v}_2\|
\ge\bigl(1 - (\eta L+\delta_k)\bigr)\|\vect{v}_1\|
\ge\tfrac34\|\vect{v}_1\|,
\]
which in turn implies the lower bound \(\|\vect{v}_2\|^2\ge\tfrac{9}{16}\|\vect{v}_1\|^2\). We weaken this to \(\|\vect{v}_2\|^2\ge\tfrac12\|\vect{v}_1\|^2\). Since \((\eta L+\delta_k)^2-1\le0\), the first term in the preceding upper bound is non-positive and can be dropped. Substituting the lower bound on \(\norm{\vect{v}_2}\) back into our main inequality, we get:
\[
\|\vect{v}_3\|^2 - \|\vect{v}_1\|^2
\;\le\;-2\,\delta_{k}\;\tfrac12\,\|\vect{v}_1\|^2
\;=\;-\,\delta_k \|\vect{v}_1\|^2,
\]
which gives a direct contraction for the norm:
\[
\|\vect{v}_3\|\le\sqrt{1 - \delta_{k}}\,\|\vect{v}_1\|
\;\le\;(1 - \tfrac{\delta_{k}}{2})\,\|\vect{v}_1\|.
\]

\smallskip
\noindent\textbf{Step 2: Drift from \(G_{k-1}\) to \(G_k\).}
Define the error at step $k+1$ as the residual:
\[
E_{k+1}\;\coloneqq\;\|G_{k}(\vect{x}_{k+1})\|\;=\;\|\vect{v}_3\|.
\]
To establish a recurrence relation for $E_k$, we first relate $G_k(\vect{x}_k)$ to $G_{k-1}(\vect{x}_k)$:
\[
G_{k}(\vect{x}_k)-G_{k-1}(\vect{x}_k)
=(\delta_{k}-\delta_{k-1})(\vect{x}_k-\vect{x}_0).
\]
Combining this with the previously derived contraction on $\|\vect{v}_3\|$ yields the following master recurrence:
\begin{equation}\label{eq:monotone-norm-recurrence}
E_{k+1}
\;\le\; \parenth{1 - \tfrac{\delta_{k}}{2}}\,
\Bigl(E_k + (\delta_{k-1}-\delta_{k})\norm{\vect{x}_k-\vect{x}_0}\Bigr),
\qquad k\ge1.
\end{equation}

\smallskip
\noindent\textbf{Step 3: Scalar recurrence.}
For the base case, we examine the initialization condition:
\begin{equation*}
E_1 = \norm{G_0(\vect{x}_1)} \le \parenth{1 - \tfrac{\delta_0}{2}} \norm{G_0(\vect{x}_0)} \le \eta L \norm{\vect{x}_0 - \vect{x}^*} = \frac{R}{8}.
\end{equation*} 
Using the boundedness from Lemma~\ref{lem:traj-bound} (\(\norm{\vect{x}_k - \vect{x}_0} \le 2R\)) and the base case \(E_1 \le R/8\), we apply Lemma~\ref{lem:recurrence} with \(\delta_k = \frac{4}{k+32}\), that is, with \(\alpha=4\), \(\beta=32\), \(\gamma=1/2\), \(c=2\), and \(c_1=1/8\). Here
\[
M=\max\left\{\frac{2\cdot4}{(1/2)4-1},\frac{32}{8}\right\}R=8R.
\]
This yields the bound:
\[
  E_{k+1} \le \frac{8}{k + 32} R.
\]

\smallskip
\noindent\textbf{Step 4: Bias removal.}
Finally, using the definition of \(G_{k-1}\) and the bound on \(E_k\), we can establish the convergence rate for the original problem:
\begin{equation*}
\norm{F(\vect{x}_k)} \le \frac{E_k + 2\delta_{k-1} R}{\eta} \le \frac{128LR}{k + 31}.
\end{equation*}


\subsection{A Regularized Past-Extragradient Method}
\label{subsec:reg-peg}

The past-extragradient (PEG) update, also known as Popov's method
\cite{popov1980modification}, is the one-new-evaluation counterpart of
extragradient in the monotone-equation setting. Instead of making two fresh
queries per iteration, PEG reuses the previous operator query in the first line
and makes one new query in the second line; in unconstrained Euclidean
saddle-point problems, closely related one-sequence forms are often called
optimistic gradient or OGDA \cite{gorbunov2022last}.

Anchored Popov and OGDA variants are now part of the accelerated monotone
operator and min-max optimization landscape, including Tran-Dinh--Luo's
Halpern-type anchored Popov scheme \cite{tran2021halpern}, APV and broader
anchored/past-call frameworks \cite{boct2024extra,tran2025accelerated}, fast
OGDA dynamics \cite{bot2025fast}, and moving-anchor Popov variants
\cite{alcala2025stochastic}. From the operator-side viewpoint in
\cref{sec:framework_principle}, the variant below is distinguished by where the
anchor is evaluated: the anchor follows the PEG operator-query points rather
than being inserted as a base-iterate correction.

The standard one-new-evaluation PEG form, written with an auxiliary sequence
\((\tilde{\vect{x}}_k)\), is
\[
    \tilde{\vect{x}}_k = \vect{x}_k - \eta F(\tilde{\vect{x}}_{k-1}),
    \qquad
    \vect{x}_{k+1} = \vect{x}_k - \eta F(\tilde{\vect{x}}_k);
\]
we use the EG-started variant with
\(\tilde{\vect{x}}_{-1}=\vect{x}_0\), so that the first step coincides with an
extragradient step. Replacing \(F\) by the scaled regularized displacement
\[
    G_k(\vect{x}) \coloneqq \eta F(\vect{x})+\delta_k(\vect{x}-\vect{x}_0)
\]
gives the regularized PEG (Reg-PEG) method:
\begin{align}
    \tilde{\vect{x}}_k &= \vect{x}_k - G_k(\tilde{\vect{x}}_{k-1}), \label{eq:reg_og_1} \\
    \vect{x}_{k+1} &= \vect{x}_k - G_k(\tilde{\vect{x}}_k), \label{eq:reg_og_2}
\end{align}
or, equivalently, after expanding \(G_k\),
\[
    \tilde{\vect{x}}_k
    =\vect{x}_k-\eta F(\tilde{\vect{x}}_{k-1})
      +\delta_k(\vect{x}_0-\tilde{\vect{x}}_{k-1}),\qquad
    \vect{x}_{k+1}
    =\vect{x}_k-\eta F(\tilde{\vect{x}}_k)
      +\delta_k(\vect{x}_0-\tilde{\vect{x}}_k).
\]
The anchor terms land at the past and current PEG operator-call points,
\(\tilde{\vect{x}}_{k-1}\) and \(\tilde{\vect{x}}_k\), rather than at the base
iterate \(\vect{x}_k\). 

\paragraph{Comparison with anchored Popov variants.}
The closest algebraic comparator is Tran-Dinh--Luo's anchored Popov scheme
\cite{tran2021halpern},
\[
    y_k
    =
    \beta_k\vect{x}_0+(1-\beta_k)\vect{x}_k-\eta_kF(y_{k-1}),
    \qquad
    \vect{x}_{k+1}
    =
    \beta_k\vect{x}_0+(1-\beta_k)\vect{x}_k-\eta_kF(y_k).
\]
Both Popov lines therefore share the same Halpern/base displacement
\(\beta_k(\vect{x}_0-\vect{x}_k)\). Nearby APV/PEAG and anchored
forward-reflected or optimistic-gradient frameworks vary the past-call direction,
splitting model, or parameter schedule, but still organize anchoring through a
base affine point or framework-level anchor
\cite{boct2024extra,tran2025accelerated}. Moving-anchor Popov variants replace
the fixed anchor by an evolving one and should be read as proximity context,
with the Popov convergence theory still less settled in that line
\cite{alcala2025stochastic}.

Reg-PEG differs at the source of the anchor. It is the ordinary PEG template
applied to the regularized displacement \(G_k\), so expanding each operator call
places the anchor at the corresponding Popov query point:
\(\delta_k(\vect{x}_0-\tilde{\vect{x}}_{k-1})\) in the past-query line and
\(\delta_k(\vect{x}_0-\tilde{\vect{x}}_k)\) in the current-query line. The
distinction claimed here is this operator-call anchor placement and the
residual-tracking proof for the resulting update, not a priority claim for
anchored Popov or OGDA acceleration as a general theme.
The theorem below is the PEG counterpart of the Reg-EG
residual bound, attained at one new \(F\)-evaluation per iteration rather than
two.

\begin{theorem}[Last-Iterate Residual Bound for Reg-PEG] \label{thm:reg_og}
Let \(F:\mathcal H\to\mathcal H\) be a monotone, \(L\)-Lipschitz operator with \(L>0\) and a solution \(\vect{x}^*\) such that \(F(\vect{x}^*) = \vect{0}\). Consider the Reg-PEG method defined by \eqref{eq:reg_og_1}--\eqref{eq:reg_og_2}, initialized with main iterate \(\vect{x}_0\) and \(\tilde{\vect{x}}_{-1}=\vect{x}_0\), with parameters \(\eta = \frac{1}{8L}\) and \(\delta_k = \frac{8}{k+64}\) for \(k \ge 0\). Let \(R \coloneqq \norm{\vect{x}_0 - \vect{x}^*}\). Then, for all \(k\ge1\), the iterates satisfy:
\[
    r_F(\vect{x}_k)=\norm{F(\vect{x}_k)} \le \frac{256LR}{k+63}.
\]
\end{theorem}

\paragraph{Proof outline.}
The proof, although more involved, follows the shared progress--drift--bias architecture, with a PEG residual potential adapted to the staggered operator-call points in \eqref{eq:reg_og_1}--\eqref{eq:reg_og_2}. First, for a fixed \(G_k\), a PEG step contracts a two-component regularized-residual potential. Second, a Lyapunov argument gives the trajectory bound \(\norm{\vect{x}_k-\vect{x}_0}\le2R\). Third, the change from \(G_{k-1}\) to \(G_k\) adds the drift term \((\delta_{k-1}-\delta_k)\norm{\vect{x}_k-\vect{x}_0}\). Fourth, the recurrence lemma bounds the regularized residual potential, and the identity \(G_{k-1}(\cdot)=\eta F(\cdot)+\delta_{k-1}(\cdot-\vect{x}_0)\) removes the bias to recover \(r_F(\vect{x}_k)\).

\subsubsection{Preparatory Lemmas}

\begin{lemma}[One-step Reg-PEG displacement bound]\label{lem:omd-dis-contract}
For the Reg-PEG update at phase \(k\), suppose that \(G_k\) is \(\delta_k\)-strongly monotone, at most \(1/4\)-Lipschitz, and has zero \(\vect{x}_k^*\). Then
\[
\norm{\vect{x}_{k+1} - \vect{x}_k^*}^2 \le (1 - \delta_k)\norm{\vect{x}_k - \vect{x}_k^*}^2 - (1-2\delta_k)\norm{\vect{x}_k - \tilde{\vect{x}}_k}^2 + \frac{1}{16}\norm{\tilde{\vect{x}}_k - \tilde{\vect{x}}_{k-1}}^2.
\]
\end{lemma}
\begin{proof}
We begin by establishing a lower bound on the term $2 \angles{G_k(\tilde{\vect{x}}_k), \tilde{\vect{x}}_k - \vect{x}_k^*}$. Using the $\delta_k$-strong monotonicity of $G_k$ and the inequality $\norm{\vect{a}+\vect{b}}^2 \le 2\norm{\vect{a}}^2 + 2\norm{\vect{b}}^2$, we have:
\begin{align*}
2 \angles{G_k(\tilde{\vect{x}}_k), \tilde{\vect{x}}_k - \vect{x}_k^*} &\ge 2\delta_k\norm{\tilde{\vect{x}}_k-\vect{x}_k^*}^2 \\
&\ge \delta_k\left(\norm{\vect{x}_k-\vect{x}_k^*}^2 - 2\norm{\vect{x}_k-\tilde{\vect{x}}_k}^2\right).
\end{align*}
Next, we derive an upper bound for the same term. From the update rule, we know that $G_k(\tilde{\vect{x}}_k) = \vect{x}_k - \vect{x}_{k+1}$. Substituting this into the expression gives:
\begin{align*}
2 \angles{G_k(\tilde{\vect{x}}_k), \tilde{\vect{x}}_k - \vect{x}_k^*} &= 2\angles{\vect{x}_k-\vect{x}_{k+1}, \tilde{\vect{x}}_k-\vect{x}_k^*} \\
&= \norm{\vect{x}_k - \vect{x}_k^*}^2 - \norm{\vect{x}_{k+1} - \vect{x}_k^*}^2 - \norm{\vect{x}_k - \tilde{\vect{x}}_k}^2 + \norm{\tilde{\vect{x}}_k - \vect{x}_{k+1}}^2.
\end{align*}
The final term in this expression can be bounded using the Lipschitz property of \(G_k\):
\[
\norm{\tilde{\vect{x}}_k - \vect{x}_{k+1}}^2 = \norm{G_k(\tilde{\vect{x}}_k) - G_k(\tilde{\vect{x}}_{k-1})}^2 \le \frac{1}{16} \norm{\tilde{\vect{x}}_k - \tilde{\vect{x}}_{k-1}}^2.
\]
Combining the lower and upper bounds and rearranging the terms yields the result.
\end{proof}

\begin{lemma}[Boundedness of the Reg-PEG trajectory]\label{lem:reg-og-traj-bound}
Let \(F:\mathcal H\to\mathcal H\) be monotone and \(L\)-Lipschitz with \(L>0\), and let \(\vect{x}^*\) satisfy \(F(\vect{x}^*)=\vect{0}\). Set \(R\coloneqq\norm{\vect{x}_0-\vect{x}^*}\). For the Reg-PEG scheme \eqref{eq:reg_og_1}--\eqref{eq:reg_og_2} with \(\eta=\frac{1}{8L}\), \(\delta_k=\frac{8}{k+64}\), and \(\tilde{\vect{x}}_{-1}=\vect{x}_0\), the iterates satisfy
\[
    \norm{\vect{x}_k-\vect{x}_0}\le 2R,\qquad k\ge0.
\]
\end{lemma}
\begin{proof}
Let \(\vect{x}_k^*\) denote the unique zero of \(G_k\). We set the step size to $\eta = \frac{1}{8L}$ and use $\delta_k \le \delta_0 = \frac{1}{8}$. This choice ensures that $G_k$ is at most $1/4$-Lipschitz for any $k \ge 0$. For \(k\ge1\), our immediate goal is to prove the following inequality, which indicates that a specific Lyapunov function is contracting:
$$
\norm{\vect{x}_{k+1} - \vect{x}_k^*}^2 + \frac{1}{16} \norm{\tilde{\vect{x}}_k - \tilde{\vect{x}}_{k-1}}^2 \le \left(1 - \delta_k\right) \left( \norm{\vect{x}_{k} - \vect{x}_k^*}^2 + \frac{1}{16} \norm{\tilde{\vect{x}}_{k-1} - \tilde{\vect{x}}_{k-2}}^2 \right).
$$
We first apply \cref{lem:omd-dis-contract}, and then relate $\norm{\vect{x}_k - \tilde{\vect{x}}_k}^2$ to quantities from the previous iteration. By the triangle inequality and the update rule:
\begin{align*}
\norm{\tilde{\vect{x}}_k - \tilde{\vect{x}}_{k-1}}^2 &\le 2\norm{\vect{x}_k - \tilde{\vect{x}}_k}^2 + 2\norm{\vect{x}_k - \tilde{\vect{x}}_{k-1}}^2 \\
&\le 2\norm{\vect{x}_k - \tilde{\vect{x}}_k}^2 + 2\norm{G_{k-1}(\tilde{\vect{x}}_{k-2}) - G_{k-1}(\tilde{\vect{x}}_{k-1})}^2 \\
&\le 2\norm{\vect{x}_k - \tilde{\vect{x}}_k}^2 + \frac{1}{8} \norm{\tilde{\vect{x}}_{k-1} - \tilde{\vect{x}}_{k-2}}^2.
\end{align*}
Rearranging this gives a lower bound on $\norm{\vect{x}_k - \tilde{\vect{x}}_k}^2$, which we substitute into the result of \cref{lem:omd-dis-contract}:
\begin{align*}
\norm{\vect{x}_{k+1} - \vect{x}_k^*}^2 \le{}& \left(1-\delta_k\right)\norm{\vect{x}_k - \vect{x}_k^*}^2 + \frac{1}{16}\norm{\tilde{\vect{x}}_k - \tilde{\vect{x}}_{k-1}}^2 \\
& - \left(1-2\delta_k\right)\left(\frac{1}{2}\norm{\tilde{\vect{x}}_k - \tilde{\vect{x}}_{k-1}}^2 - \frac{1}{16}\norm{\tilde{\vect{x}}_{k-1} - \tilde{\vect{x}}_{k-2}}^2\right) \\
=& \left(1-\delta_k\right)\norm{\vect{x}_k - \vect{x}_k^*}^2 + \left[\frac{1}{16} - \frac{1}{2}\left(1- 2 \delta_k\right)\right]\norm{\tilde{\vect{x}}_k - \tilde{\vect{x}}_{k-1}}^2 \\
& + \frac{1}{16}\left(1-2\delta_k\right)\norm{\tilde{\vect{x}}_{k-1} - \tilde{\vect{x}}_{k-2}}^2.
\end{align*}
After rearranging, we get:
\[
\begin{aligned}
&\norm{\vect{x}_{k+1} - \vect{x}_k^*}^2
+ \parenth{\frac{7}{16} - \delta_k}
    \norm{\tilde{\vect{x}}_k - \tilde{\vect{x}}_{k-1}}^2 \\
&\qquad\le
\left(1-\delta_k\right)\norm{\vect{x}_k - \vect{x}_k^*}^2
+ \frac{1}{16}\left(1-2\delta_k\right)
    \norm{\tilde{\vect{x}}_{k-1} - \tilde{\vect{x}}_{k-2}}^2.
\end{aligned}
\]
Given our choice of $\delta_k \le 1/8$, we have \((7/16-\delta_k)\ge1/16\) and \((1-2\delta_k)/16\le(1-\delta_k)/16\). Therefore the previous display implies
\[
\norm{\vect{x}_{k+1} - \vect{x}_k^*}^2
+\frac{1}{16}\norm{\tilde{\vect{x}}_k - \tilde{\vect{x}}_{k-1}}^2
\le
(1-\delta_k)\left(
\norm{\vect{x}_{k} - \vect{x}_k^*}^2
+\frac{1}{16}\norm{\tilde{\vect{x}}_{k-1} - \tilde{\vect{x}}_{k-2}}^2
\right).
\]
This allows us to define the Lyapunov function. For \(j\ge1\), set
\[
B_j\coloneqq \sqrt{\norm{\vect{x}_{j} - \vect{x}_{j-1}^*}^2 + \frac{1}{16} \norm{\tilde{\vect{x}}_{j-1} - \tilde{\vect{x}}_{j-2}}^2}.
\]
Then for \(k\ge1\),
\begin{align*}
B_{k+1}
&\le \sqrt{\left(1 - \delta_k\right) \left( \norm{\vect{x}_{k} - \vect{x}_k^*}^2 + \frac{1}{16} \norm{\tilde{\vect{x}}_{k-1} - \tilde{\vect{x}}_{k-2}}^2 \right)} \\
&\le \parenth{1 - \frac{\delta_k}{2}} \sqrt{ \norm{\vect{x}_{k} - \vect{x}_k^*}^2 + \frac{1}{16} \norm{\tilde{\vect{x}}_{k-1} - \tilde{\vect{x}}_{k-2}}^2} \\
&\le \parenth{1 - \frac{\delta_k}{2}} \parenth{\sqrt{ \norm{\vect{x}_{k} - \vect{x}_{k-1}^*}^2 + \frac{1}{16} \norm{\tilde{\vect{x}}_{k-1} - \tilde{\vect{x}}_{k-2}}^2} + \norm{\vect{x}_{k-1}^* - \vect{x}_k^*}}.
\end{align*}
Here, the second inequality uses $\sqrt{1 - x} \le 1 - x/2$ for $x \in [0,1]$, and the third inequality uses the triangle inequality property $\sqrt{\norm{\vect{a} + \vect{b}}^2 + \norm{\vect{c}}^2} \le \sqrt{\norm{\vect{a}}^2 + \norm{\vect{c}}^2} + \norm{\vect{b}}$. Applying the bound on the stability of the regularized solution from \cref{lem:stability-reg}, we have:
\begin{equation}
    B_{k+1} \le \parenth{1 - \frac{\delta_k}{2}} \parenth{B_k + \frac{\delta_{k-1} - \delta_k}{\delta_{k-1}} R}.
\end{equation}
We now use induction to show that the iterates are bounded, i.e., $B_k \le R$. For the base case $k = 1$, we use the \(k=0\) instance of \cref{lem:omd-dis-contract}, add the term \(\frac1{16}\norm{\tilde{\vect{x}}_0-\tilde{\vect{x}}_{-1}}^2\), and use \(\tilde{\vect{x}}_{-1}=\vect{x}_0\). Together with \cref{lem:reg-and-x0}, this gives
\begin{align*}
    B_1^2
    &= \norm{\vect{x}_{1} - \vect{x}_0^*}^2 + \frac{1}{16} \norm{\tilde{\vect{x}}_0 - \tilde{\vect{x}}_{-1}}^2  \\
    &\le (1 - \delta_0)\norm{\vect{x}_0 - \vect{x}_0^*}^2 - (1 - 2\delta_0) \norm{\vect{x}_0 - \tilde{\vect{x}}_0}^2 + \frac{1}{8}\norm{\vect{x}_0 - \tilde{\vect{x}}_0}^2  \\
    &\le \norm{\vect{x}_0 - \vect{x}_0^*}^2 \le R^2.
\end{align*}
Assuming $B_k \le R$ holds, and writing $\delta_k = \frac{c}{k + 8c}$ with \(c=8\), we use
\[
\frac{\delta_{k-1}-\delta_k}{\delta_{k-1}}=\frac{1}{k+8c},
\qquad
\frac{\delta_k}{2}=\frac{c}{2(k+8c)}\ge\frac{1}{k+8c},
\]
where the last inequality follows from the weaker bound \(c\ge2\). Therefore
\begin{align*}
    B_{k+1} \le \parenth{1 - \frac{1}{k + 8c}} \parenth{R + \frac{1}{k + 8c} R} \le R.
\end{align*}
This completes the induction. The case \(k=0\) is immediate because \(\norm{\vect{x}_0-\vect{x}_0}=0\). For \(k\ge0\), the induction bound at index \(k+1\) implies that the trajectory is bounded:
\begin{equation}
    \norm{\vect{x}_{k+1} - \vect{x}_0} \le \norm{\vect{x}_{k+1} - \vect{x}_k^*} + \norm{\vect{x}_k^* - \vect{x}_0} \le R+R = 2R.
\end{equation}
\end{proof}

\begin{lemma}[PEG Residual-Potential Contraction]\label{lem:peg-contraction}
Let $\vect{a},\vect{b},\vect{c},\vect{d}\in\mathcal H$ and suppose $0 < \delta \le \frac{1}{3}$. We assume the following two conditions hold:
\begin{align}
\langle \vect{b}-\vect{a}, \vect{c}\rangle &\ge \delta \norm{\vect{c}}^{2}, && \text{(strong monotonicity)}\label{eq:monotone}\\
\norm{\vect{a}-\vect{c}} &\le \frac{1}{3}\norm{\vect{c}-\vect{d}}, && \text{(Lipschitzness)}\label{eq:lipschitz}
\end{align}
Define a potential function
$\Phi(\vect{x}, \vect{y}) := \norm{\vect{x}}^{2} + 2\norm{\vect{x}-\vect{y}}^{2}$. Then
\begin{equation}
\Phi(\vect{a}, \vect{c}) \le \bigl(1 - \tfrac{\delta}{2}\bigr) \Phi(\vect{b}, \vect{d}).\label{eq:goal}
\end{equation}
\end{lemma}

\begin{proof}
The proof proceeds by establishing a lower bound on the difference $\Phi(\vect{b}, \vect{d}) - \Phi(\vect{a}, \vect{c})$ and showing that this is greater than or equal to $\frac{\delta}{2}\Phi(\vect{b}, \vect{d})$.

Expand the difference between the potentials:
\begin{align*}
\Phi(\vect{b}, \vect{d}) - \Phi(\vect{a}, \vect{c}) &= \bigl(\norm{\vect{b}}^{2}-\norm{\vect{a}}^{2}\bigr) + 2\bigl(\norm{\vect{b}-\vect{d}}^{2}-\norm{\vect{a}-\vect{c}}^{2}\bigr).
\end{align*}
We first find a lower bound for the term $\norm{\vect{b}}^{2}-\norm{\vect{a}}^{2}$. By decomposing $\vect{a}$ as $(\vect{a}-\vect{c})+\vect{c}$ and invoking the strong monotonicity condition from \cref{eq:monotone}, we obtain:
\begin{align*}
\norm{\vect{b}}^{2}-\norm{\vect{a}}^{2} &= \norm{\vect{b}-\vect{a}}^{2}+2\langle \vect{a}, \vect{b}-\vect{a}\rangle \\
&= \norm{\vect{b}-\vect{a}}^{2}+2\langle \vect{a}-\vect{c}, \vect{b}-\vect{a}\rangle + 2\langle \vect{c}, \vect{b}-\vect{a}\rangle \\
&\ge \norm{\vect{b}-\vect{a}}^{2}+2\langle \vect{a}-\vect{c}, \vect{b}-\vect{a}\rangle + 2\delta \norm{\vect{c}}^{2}.
\end{align*}
Using the identity $\norm{\vect{x}}^2 + 2\langle \vect{y}, \vect{x} \rangle = \norm{\vect{x}+\vect{y}}^2 - \norm{\vect{y}}^2$ with $\vect{x} = \vect{b}-\vect{a}$ and $\vect{y} = \vect{a}-\vect{c}$, the first two terms simplify to $\norm{\vect{b}-\vect{c}}^{2}-\norm{\vect{a}-\vect{c}}^{2}$. This gives:
\[
\norm{\vect{b}}^{2}-\norm{\vect{a}}^{2} \ge \norm{\vect{b}-\vect{c}}^{2}-\norm{\vect{a}-\vect{c}}^{2}+2\delta \norm{\vect{c}}^{2}.
\]
Substituting this result back into the expression for the potential difference yields:
\begin{equation}\label{eq:gf-first}
\Phi(\vect{b}, \vect{d}) - \Phi(\vect{a}, \vect{c}) \ge \norm{\vect{b}-\vect{c}}^{2} + 2\norm{\vect{b}-\vect{d}}^{2} - 3\norm{\vect{a}-\vect{c}}^{2} + 2\delta \norm{\vect{c}}^{2}.
\end{equation}
Next, we apply the Lipschitz bound from \cref{eq:lipschitz} to handle the negative term involving $\vect{a}-\vect{c}$:
\begin{equation}\label{eq:gf-second}
\Phi(\vect{b}, \vect{d}) - \Phi(\vect{a}, \vect{c}) \ge \norm{\vect{b}-\vect{c}}^{2} + 2\norm{\vect{b}-\vect{d}}^{2} - \frac{1}{3}\norm{\vect{d}-\vect{c}}^{2} + 2\delta \norm{\vect{c}}^{2}.
\end{equation}

Our next step is to relate this lower bound to $\Phi(\vect{b}, \vect{d})$. We establish an upper bound for $\frac{\delta}{2} \Phi(\vect{b}, \vect{d})$ by applying the triangle inequality in the form $\norm{\vect{x}+\vect{y}}^2 \le 2(\norm{\vect{x}}^2 + \norm{\vect{y}}^2)$, which gives $\norm{\vect{b}}^2 = \norm{(\vect{b}-\vect{c})+\vect{c}}^2 \le 2(\norm{\vect{b}-\vect{c}}^2 + \norm{\vect{c}}^2)$. Therefore,
\begin{equation}\label{eq:half-g}
\frac{\delta}{2} \Phi(\vect{b}, \vect{d}) = \frac{\delta}{2}\bigl(\norm{\vect{b}}^{2}+2\norm{\vect{b}-\vect{d}}^{2}\bigr) \le \delta\bigl(\norm{\vect{b}-\vect{c}}^{2}+\norm{\vect{c}}^{2}+\norm{\vect{b}-\vect{d}}^{2}\bigr).
\end{equation}

Finally, we combine our results. Subtracting the bound in \cref{eq:half-g} from the inequality in \cref{eq:gf-second} yields:
\begin{align*}
\Phi(\vect{b}, \vect{d}) - \Phi(\vect{a}, \vect{c}) - \tfrac{\delta}{2} \Phi(\vect{b}, \vect{d}) &\ge (1-\delta)\norm{\vect{b}-\vect{c}}^{2} + (2-\delta)\norm{\vect{b}-\vect{d}}^{2} - \frac{1}{3}\norm{\vect{d}-\vect{c}}^{2} + \delta \norm{\vect{c}}^{2}.
\end{align*}
To simplify the right-hand side, we again use the triangle inequality:
\[
\norm{\vect{d}-\vect{c}}^{2}
=
\norm{(\vect{d}-\vect{b})+(\vect{b}-\vect{c})}^2
\le
2\norm{\vect{b}-\vect{d}}^{2}
+2\norm{\vect{b}-\vect{c}}^{2}.
\]
This leads to:
\begin{align*}
\bigl(1 - \tfrac{\delta}{2}\bigr)\Phi(\vect{b}, \vect{d}) - \Phi(\vect{a}, \vect{c}) &\ge \bigl(1-\delta - \tfrac{2}{3}\bigr)\norm{\vect{b}-\vect{c}}^{2} + \bigl(2-\delta - \tfrac{2}{3}\bigr)\norm{\vect{b}-\vect{d}}^{2}+\delta \norm{\vect{c}}^{2} \\
&= \bigl(\tfrac{1}{3}-\delta\bigr)\norm{\vect{b}-\vect{c}}^{2} + \bigl(\tfrac{4}{3}-\delta\bigr)\norm{\vect{b}-\vect{d}}^{2}+\delta \norm{\vect{c}}^{2}.
\end{align*}
Since we assumed $0 < \delta \le \frac{1}{3}$, all the coefficients on the right-hand side are non-negative, which implies the entire expression is non-negative. Thus,
\[
\bigl(1 - \tfrac{\delta}{2}\bigr)\Phi(\vect{b}, \vect{d}) - \Phi(\vect{a}, \vect{c}) \ge 0,
\]
which rearranges to $\Phi(\vect{a}, \vect{c}) \le \parenth{1 - \delta/2}\Phi(\vect{b}, \vect{d})$, the desired result.
\end{proof}

\begin{lemma}[Two-component norm perturbation]\label{lem:omd-sqrt-perturb}
Let $(\mathcal{X},\|\cdot\|)$ be a normed space.
For any $\vect{a},\vect{b},\vect{c},\vect{d}\in\mathcal{X}$ with $\|\vect{b}\|\le\|\vect{d}\|$,
\[
\sqrt{\|\vect{a}\|^{2}+\|\vect{b}\|^{2}}
\;\le\;
\sqrt{\|\vect{c}\|^{2}+\|\vect{d}\|^{2}}
+\|\vect{a}-\vect{c}\|.
\]
\end{lemma}

\begin{proof}
Since $\|\vect{b}\|\le\|\vect{d}\|$,
\[
\sqrt{\|\vect{a}\|^{2}+\|\vect{b}\|^{2}}
\le
\sqrt{\|\vect{a}\|^{2}+\|\vect{d}\|^{2}}.
\]
We can view the pair of scalar norms $(\|\vect{a}\|,\|\vect{d}\|)$ as a vector in $\mathbb{R}^{2}$ and write
\[
(\|\vect{a}\|,\|\vect{d}\|)
=
(\|\vect{c}\|,\|\vect{d}\|)+(\|\vect{a}\|-\|\vect{c}\|,0).
\]
Applying the standard triangle inequality for vectors in $\mathbb{R}^{2}$ gives
\[
\sqrt{\|\vect{a}\|^{2}+\|\vect{d}\|^{2}}
\le
\sqrt{\|\vect{c}\|^{2}+\|\vect{d}\|^{2}}
+|\|\vect{a}\|-\|\vect{c}\||
\le
\sqrt{\|\vect{c}\|^{2}+\|\vect{d}\|^{2}}
+\|\vect{a}-\vect{c}\|,
\]
where the last step follows from the reverse triangle inequality. Combining the two displayed inequalities yields the claim.
\end{proof}

\subsubsection{Proof of Theorem~\ref{thm:reg_og}}

Before the algebra, we record the bookkeeping. The fixed-\(G_k\) PEG step contracts a two-component residual potential at the staggered pair \((\vect{x}_{k+1},\tilde{\vect{x}}_k)\). Then \cref{lem:omd-sqrt-perturb} compares the \(G_k\)-potential at \(\vect{x}_k\) with the previous \(G_{k-1}\)-potential, losing only the drift term \((\delta_{k-1}-\delta_k)\norm{\vect{x}_k-\vect{x}_0}\). The trajectory bound turns this drift into \(2(\delta_{k-1}-\delta_k)R\), after which the scalar recurrence lemma gives the displayed residual rate.

Define
\[
E_0 \coloneqq
\sqrt{\|G_0(\vect{x}_0)\|^2
      +2\|G_0(\vect{x}_0)-G_0(\tilde{\vect{x}}_{-1})\|^2},
\]
and, for \(k\ge1\),
\[
E_k \coloneqq
\sqrt{\|G_{k-1}(\vect{x}_k)\|^2
      +2\|G_{k-1}(\vect{x}_k)-G_{k-1}(\tilde{\vect{x}}_{k-1})\|^2}.
\]

\smallskip
\noindent\textbf{Step 1: Fixed-operator residual-potential contraction.}
For a fixed \(k\), set
\[
    \vect{a}=G_k(\vect{x}_{k+1}),\quad
    \vect{b}=G_k(\vect{x}_k),\quad
    \vect{c}=G_k(\tilde{\vect{x}}_k),\quad
    \vect{d}=G_k(\tilde{\vect{x}}_{k-1}).
\]
The update rules give \(\vect{c}=\vect{x}_k-\vect{x}_{k+1}\) and \(\tilde{\vect{x}}_k=\vect{x}_k-\vect{d}\). Hence
\[
\vect{d}-\vect{c}
=
(\vect{x}_k-\tilde{\vect{x}}_k)-(\vect{x}_k-\vect{x}_{k+1})
=
\vect{x}_{k+1}-\tilde{\vect{x}}_k.
\]
Strong monotonicity of \(G_k\) gives
\[
    \angles{\vect{b}-\vect{a},\vect{c}}
    =
    \angles{G_k(\vect{x}_k)-G_k(\vect{x}_{k+1}),\vect{x}_k-\vect{x}_{k+1}}
    \ge \delta_k\norm{\vect{c}}^2.
\]
Since \(\delta_k\le\delta_0=1/8<1/3\) and \(\eta L+\delta_k\le1/4\), Lipschitzness gives
\[
    \norm{\vect{a}-\vect{c}}
    \le
    (\eta L+\delta_k)\norm{\vect{x}_{k+1}-\tilde{\vect{x}}_k}
    \le
    \frac14\norm{\vect{d}-\vect{c}}
    \le
    \frac13\norm{\vect{d}-\vect{c}}.
\]
Applying \cref{lem:peg-contraction} yields
\[
\begin{aligned}
&\|G_k(\vect{x}_{k+1})\|^2
  +2\|G_k(\vect{x}_{k+1})-G_k(\tilde{\vect{x}}_k)\|^2 \\
&\qquad\le
\left(1-\frac{\delta_k}{2}\right)
\left(
\|G_k(\vect{x}_k)\|^2
+2\|G_k(\vect{x}_k)-G_k(\tilde{\vect{x}}_{k-1})\|^2
\right).
\end{aligned}
\]
The left-hand side is \(E_{k+1}^2\) by the staggered definition of \(E_{k+1}\).
Taking square roots and using \(\sqrt{1-z}\le1-z/2\) with \(z=\delta_k/2\) gives
\[
E_{k+1}
\le
\left(1-\frac{\delta_k}{4}\right)
\sqrt{
\|G_k(\vect{x}_k)\|^2
+2\|G_k(\vect{x}_k)-G_k(\tilde{\vect{x}}_{k-1})\|^2}.
\]

\smallskip
\noindent\textbf{Step 2: Drift from \(G_{k-1}\) to \(G_k\).}
It remains to compare the \(G_k\)-potential at \(\vect{x}_k\) with \(E_k\). By monotonicity of \(F\) and \(\delta_k\le\delta_{k-1}\),
\begin{align*}
&2\norm{G_k(\vect{x}_k)-G_k(\tilde{\vect{x}}_{k-1})}^2
-2\norm{G_{k-1}(\vect{x}_k)-G_{k-1}(\tilde{\vect{x}}_{k-1})}^2 \\
&=
4\eta(\delta_k-\delta_{k-1})
\angles{F(\vect{x}_k)-F(\tilde{\vect{x}}_{k-1}),\vect{x}_k-\tilde{\vect{x}}_{k-1}}
{}+
2(\delta_k^2-\delta_{k-1}^2)
\norm{\vect{x}_k-\tilde{\vect{x}}_{k-1}}^2
\le0.
\end{align*}
Indeed, with \(\vect{u}=F(\vect{x}_k)-F(\tilde{\vect{x}}_{k-1})\) and \(\vect{w}=\vect{x}_k-\tilde{\vect{x}}_{k-1}\), monotonicity gives \(\angles{\vect{u},\vect{w}}\ge0\). Since \(\delta_k-\delta_{k-1}\le0\) and \(\delta_k^2-\delta_{k-1}^2\le0\), both terms in the displayed difference are non-positive.
Moreover,
\[
    \norm{G_k(\vect{x}_k)-G_{k-1}(\vect{x}_k)}
    =
    (\delta_{k-1}-\delta_k)\norm{\vect{x}_k-\vect{x}_0}.
\]
For the notation of \cref{lem:omd-sqrt-perturb}, take
\[
\begin{aligned}
\vect{a}&=G_k(\vect{x}_k),&
\vect{c}&=G_{k-1}(\vect{x}_k),\\
\vect{b}&=\sqrt{2}\big(G_k(\vect{x}_k)-G_k(\tilde{\vect{x}}_{k-1})\big),&
\vect{d}&=\sqrt{2}\big(G_{k-1}(\vect{x}_k)-G_{k-1}(\tilde{\vect{x}}_{k-1})\big).
\end{aligned}
\]
The preceding display gives \(\norm{\vect{b}}\le\norm{\vect{d}}\), and the drift bound gives \(\norm{\vect{a}-\vect{c}}\le(\delta_{k-1}-\delta_k)\norm{\vect{x}_k-\vect{x}_0}\). Applying \cref{lem:omd-sqrt-perturb} to the two-component norm gives, for \(k\ge1\),
\[
    E_{k+1}
    \le
    \left(1-\frac{\delta_k}{4}\right)
    \left(
        E_k+(\delta_{k-1}-\delta_k)\norm{\vect{x}_k-\vect{x}_0}
    \right).
\]
By \cref{lem:reg-og-traj-bound},
\[
    E_{k+1}
    \le
    \left(1-\frac{\delta_k}{4}\right)
    \left(E_k+2(\delta_{k-1}-\delta_k)R\right),
    \qquad k\ge1.
\]

\smallskip
\noindent\textbf{Step 3: Scalar recurrence.}
For the base case, since \(\tilde{\vect{x}}_{-1}=\vect{x}_0\),
\[
    E_0=\norm{G_0(\vect{x}_0)}
    =\eta\norm{F(\vect{x}_0)}
    \le \eta L R=\frac{R}{8}.
\]
Because \(\tilde{\vect{x}}_{-1}=\vect{x}_0\), the right-hand potential in the fixed-operator contraction at \(k=0\) is exactly \(E_0\). Hence \(E_1\le(1-\delta_0/4)E_0\le E_0\le R/8\). Applying \cref{lem:recurrence} with
\[
    \delta_k=\frac{8}{k+64},\qquad
    \alpha=8,\qquad
    \beta=64,\qquad
    \gamma=\frac14,\qquad
    c=2,\qquad
    c_1=\frac18
\]
gives
\[
    E_k\le \frac{16R}{k+63},\qquad k\ge1.
\]

\smallskip
\noindent\textbf{Step 4: Bias removal.}
Finally, for \(k\ge1\),
\[
\norm{F(\vect{x}_k)}
\le
\frac{E_k+\delta_{k-1}\norm{\vect{x}_k-\vect{x}_0}}{\eta}
\le
\frac{E_k+2\delta_{k-1}R}{\eta}
\le
\frac{256LR}{k+63}.
\]


\section{Conclusion}
\label{sec:conclusion}

Anchoring is often presented as a structural change to the base algorithm, and the role of the anchor and its placement can be unintuitive, even when the resulting rates are sharp. This work focuses on an operator-side view of anchoring for fixed-point and monotone-equation methods: apply a vanishing Tikhonov term to the queried operator or displacement, then run the standard Picard, forward, extragradient, or Popov past-extragradient template on the resulting moving regularized object. A single analysis framework captures the intuition: the method makes more rapid progress on a better conditioned problem, while the parameters balance that progress against the bias with respect to the original problem. More importantly, this perspective clarifies the mechanism of anchoring ``acceleration'' and has directly led to several new results presented in this paper.

The results are established in a simple single-valued unconstrained setting to keep the mechanism transparent. Extensions to monotone inclusions and variational inequalities are natural follow-ups.

\appendix
\numberwithin{lemma}{section}
\numberwithin{proposition}{section}
\section{A Recurrence Lemma}\label{app:technical-lemma}
We use the following deterministic recurrence lemma throughout the manuscript.

\begin{lemma}\label{lem:recurrence}
Let \(c,c_1,\alpha,\beta,\gamma>0\) and \(R\ge0\) with $1 < \gamma \alpha \le \beta$ and define
\[
   \delta_n \;=\; \frac{\alpha}{\,n+\beta\,}, 
   \qquad n\ge0.
\]
Assume the sequence \((e_n)_{n\ge 1}\) satisfies $e_1 \le c_1\,R$ and
\begin{equation} \label{eq:rec}
   e_{n+1}  \le \bigl(1-\gamma\delta_n\bigr)
         \Bigl(e_{n}+c\,(\delta_{n-1}-\delta_n)\,R\Bigr),
         \qquad n\ge 1. 
\end{equation}
Then for every \(n\ge 1\),
\[
   e_n \le \frac{M}{n + \beta - 1}, \quad \text{where }
        M := R\max\!\bigl\{
            \tfrac{c\,\alpha}{\gamma\alpha-1},
            \;c_1\beta
         \bigr\}.
\]
\end{lemma}

\begin{proof}
We prove by induction that \(e_n\le M/(n+\beta-1)\) for all \(n\ge 1\).

\smallskip
\noindent\textbf{Base case (\(n=1\)).}
By the definition of \(M\) we have
\[
   e_1 \;\le\; c_1R \;\le\; \frac{M}{\beta}.
\]

\smallskip
\noindent\textbf{Inductive step.}
Assume \(e_{n}\le M/(n-1+\beta)\) for some \(n\ge 1\).
Since
\(
   \delta_{n-1}-\delta_n
   =\dfrac{\alpha}{(n-1+\beta)(n+\beta)},
\)
recurrence \eqref{eq:rec} yields
\[
\begin{aligned}
   e_{n+1}
   &\le
   \Bigl(1-\frac{\gamma\alpha}{n+\beta}\Bigr)
   \Bigl(
      \frac{M}{n-1+\beta}
      +\frac{c\alpha R}{(n-1+\beta)(n+\beta)}
   \Bigr) \\[4pt]
   &=
   \frac{M(1-\gamma\alpha)+c\alpha R
         -\dfrac{\gamma\alpha\,c\alpha R}{n+\beta}}
        {(n-1+\beta)(n+\beta)}
   \;+\;
   \frac{M}{n+\beta}.
\end{aligned}
\]
By the definition of \(M\),
\[
   M \;\ge\; \frac{c\alpha R}{\gamma\alpha-1}
   \;\Longrightarrow\;
   M(1-\gamma\alpha)+c\alpha R \;\le\; 0.
\]
Hence the first fraction is non-positive, giving
\(
   e_{n+1}\le M/(n+\beta).
\)
Thus the inductive step closes and the claimed bound holds for all \(n\ge 1\), implying
\(e_n=O(1/n)\).
\end{proof}

\section{Fixed-Point Extensions}\label{app:fixed-point-extensions}

\subsection{An Episodic Warm-Restart Scheme}
Here we present an episodic fixed-point variant of the tracking
argument. From the operator-side viewpoint, this is a natural continuation
scheme: for a fixed value of \(\delta_k\), the map \(T_k\) is contractive, so
one may take several Picard steps on the current regularized problem and then
warm-start the next, less regularized phase. The proposition below measures the
resulting residual in terms of the total number of inner Picard steps.

Formally, at each phase \(k\), apply the operator \(T_k\) for \(t_k\) iterations, initialized from the output of the previous phase:
\begin{equation}
    \vect{x}_{k+1} = (T_k)^{t_k}(\vect{x}_k)
    \equiv
    \underbrace{T_k \circ \cdots \circ T_k}_{t_k \text{ times}}(\vect{x}_k).
\end{equation}

\begin{proposition}[Episodic Halpern Scheme with Geometric Rate]\label{prop:episodic_halpern}
Let \(T\) be a nonexpansive mapping with \(\operatorname{Fix}(T) \ne \varnothing\). Consider the episodic scheme where
\[
    \vect{x}_{k+1} = (T_k)^{t_k}(\vect{x}_k),
    \qquad
    T_k(\vect{x}) \coloneqq (1-\delta_k) T(\vect{x}) + \delta_k \vect{x}_0.
\]
Fix any \(\vect{x}^*\in\operatorname{Fix}(T)\) and set \(R\coloneqq\norm{\vect{x}_0-\vect{x}^*}\). With \(\delta_k = 2^{-k}\) and \(t_k = \lceil \ln(4)/\delta_k \rceil\), the residual norm after \(K\) phases satisfies
\[
    \norm{\vect{x}_K - T(\vect{x}_K)} \le 12R\,2^{-K},
    \qquad K\ge1.
\]
Furthermore, in terms of the total number of inner iterations \(N_K = \sum_{j=0}^{K-1} t_j\), this gives
\[
    \norm{\vect{x}_K - T(\vect{x}_K)} = O\left(\frac{R}{N_K}\right).
\]
\end{proposition}

\begin{proof}
The analysis follows the same structure as the Halpern proof in the main text. The ball \(B(\vect{x}^*,R)\) is invariant under each \(T_k\): if \(\norm{\vect{x}-\vect{x}^*}\le R\), then
\[
\norm{T_k(\vect{x})-\vect{x}^*}
\le
(1-\delta_k)\norm{T(\vect{x})-\vect{x}^*}
+\delta_k\norm{\vect{x}_0-\vect{x}^*}
\le R.
\]
Thus \(\vect{x}_k\in B(\vect{x}^*,R)\) for all \(k\), and in particular \(\norm{T(\vect{x}_k)-\vect{x}_0}\le2R\). For \(k\ge1\), let \(E_k \coloneqq \norm{\vect{x}_k - T_{k-1}(\vect{x}_k)}\). If \(S\) is a \(q\)-contraction, then \(\norm{S^t\vect{x}-S^{t+1}\vect{x}}\le q^t\norm{\vect{x}-S\vect{x}}\). Applying this with \(S=T_k\) gives
\[
    E_{k+1} \le (1-\delta_k)^{t_k} \left(E_k + 2(\delta_{k-1} - \delta_k)R\right),
    \qquad k\ge1.
\]
For \(\delta_k = 2^{-k}\), we have \(\delta_{k-1} - \delta_k = \delta_k\) for \(k \ge 1\). Setting \(t_k = \lceil \ln(4)/\delta_k \rceil\) gives \((1-\delta_k)^{t_k} \le e^{-\delta_k t_k} \le 1/4\), and therefore
\[
    E_{k+1} \le \frac{1}{4}(E_k + 2\delta_k R).
\]
Since \(\delta_0=1\), \(T_0\) is the constant map \(\vect{x}\mapsto \vect{x}_0\), so \(E_1=0\le 4R\delta_0\). We prove by induction that \(E_k\le 4R\delta_{k-1}\) for all \(k\ge1\). If \(E_{k} \le 4R\delta_{k-1}\), then
\[
    E_{k+1} \le \frac{1}{4}(4R\delta_{k-1} + 2R\delta_k)
    = \frac{1}{4}(8R\delta_k + 2R\delta_k)
    = \frac{5}{2}R\delta_k \le 4R\delta_k.
\]
Hence, by induction,
\[
    \norm{\vect{x}_k - T(\vect{x}_k)}
    \le E_k + \delta_{k-1}\norm{\vect{x}_0 - T(\vect{x}_k)}
    \le 4R\delta_{k-1} + 2R\delta_{k-1}
    = 6R\delta_{k-1}.
\]
At \(k=K\ge1\), this gives \(\norm{\vect{x}_K - T(\vect{x}_K)}\le 6R\delta_{K-1}=12R2^{-K}\). Also,
\[
N_K=\sum_{j=0}^{K-1}\left\lceil (\ln 4)2^j\right\rceil
\le (\ln4+1)(2^K-1),
\]
so \(N_K+\ln4+1\le(\ln4+1)2^K\), or equivalently \(2^{-K}\le(\ln4+1)/(N_K+\ln4+1)\). The same bound is therefore \(O(R/N_K)\).
\end{proof}

\subsection{A Viscosity Variant}\label{app:viscosity}
Here we present the viscosity generalization of the same fixed-point
regularization. The fixed anchor \(\vect{x}_0\) is the special case of a
constant contraction \(g\equiv\vect{x}_0\). Viscosity approximation replaces
this constant map by a general contraction \(g\), a standard generalization of
Halpern/Browder fixed-point regularization
\cite{moudafi2000viscosity,xu2004viscosity}. The proposition below shows that
the same moving-contraction tracking proof applies to this form. Given a map
\(g:\mathcal X\to\mathcal X\) and parameters \((\delta_k)\), define
\begin{equation}\label{eq:halpern_viscosity}
    T_k(\vect{x}) \coloneqq (1-\delta_k)T(\vect{x})+\delta_k g(\vect{x}).
\end{equation}

\begin{proposition}[Viscosity residual bound]\label{prop:halpern_viscosity}
Assume the following setting. Let \(T\) be a nonexpansive mapping with \(\operatorname{Fix}(T)\ne\varnothing\), and let \(g\) be a contraction with factor \((1-\rho)\) for some \(\rho\in(0,1]\). Let \((\vect{x}_k)_{k\ge0}\) be generated by \(\vect{x}_{k+1}=T_k(\vect{x}_k)\) with
\[
    \delta_k=\frac{2}{\rho k+2}.
\]
For any $\vect{x}^*\in\operatorname{Fix}(T)$, define
\[
    R\coloneqq \max\braces{\norm{\vect{x}_0-\vect{x}^*},\,\frac{\norm{g(\vect{x}^*)-\vect{x}^*}}{\rho}}.
\]
Then, for all $k\ge1$,
\[
    \norm{\vect{x}_k-T(\vect{x}_k)}
    \le
    \frac{8R}{\rho k+2-\rho}.
\]
\end{proposition}

\begin{proof}
Since $\delta_0=1$ and $(\delta_k)$ is non-increasing, $\delta_k\in(0,1]$ for all $k$. Hence $T_k$ is a contraction with factor $1-\rho\delta_k$:
\[
    \norm{T_k(\vect{x})-T_k(\vect{y})}
    \le (1-\delta_k)\norm{\vect{x}-\vect{y}}
       +\delta_k(1-\rho)\norm{\vect{x}-\vect{y}}
    =(1-\rho\delta_k)\norm{\vect{x}-\vect{y}}.
\]
We first show boundedness. If $\norm{\vect{x}-\vect{x}^*}\le R$, then
\begin{align*}
    \norm{T_k(\vect{x})-\vect{x}^*}
    &\le (1-\delta_k)\norm{T(\vect{x})-T(\vect{x}^*)}
       +\delta_k\norm{g(\vect{x})-\vect{x}^*} \\
    &\le (1-\delta_k)\norm{\vect{x}-\vect{x}^*}
       +\delta_k(1-\rho)\norm{\vect{x}-\vect{x}^*}
       +\delta_k\norm{g(\vect{x}^*)-\vect{x}^*} \\
    &\le (1-\rho\delta_k)R+\rho\delta_k R
    =R.
\end{align*}
Thus $\norm{\vect{x}_k-\vect{x}^*}\le R$ for all $k$ by induction.

Define $E_k\coloneqq\norm{\vect{x}_k-T_{k-1}(\vect{x}_k)}$ for $k\ge1$. The contraction of $T_k$ gives
\[
    E_{k+1}\le (1-\rho\delta_k)\norm{\vect{x}_k-T_k(\vect{x}_k)}.
\]
Moreover,
\begin{align*}
    \norm{\vect{x}_k-T_k(\vect{x}_k)}
    &\le E_k+\norm{T_{k-1}(\vect{x}_k)-T_k(\vect{x}_k)} \\
    &\le E_k+(\delta_{k-1}-\delta_k)\norm{T(\vect{x}_k)-g(\vect{x}_k)}.
\end{align*}
The boundedness estimate implies
\[
    \norm{T(\vect{x}_k)-g(\vect{x}_k)}
    \le R+(1-\rho)R+\rho R=2R.
\]
Therefore
\[
    E_{k+1}\le (1-\rho\delta_k)\left(E_k+2(\delta_{k-1}-\delta_k)R\right).
\]
Also \(E_1\le2R\). Indeed, since \(\delta_0=1\), we have \(\vect{x}_1=T_0(\vect{x}_0)=g(\vect{x}_0)\); the invariant-ball argument gives \(\vect{x}_1\in B(\vect{x}^*,R)\), and applying it once more gives \(T_0(\vect{x}_1)=g(\vect{x}_1)\in B(\vect{x}^*,R)\). Therefore \(E_1=\norm{\vect{x}_1-T_0(\vect{x}_1)}\le2R\). Applying \cref{lem:recurrence} with
\[
    \alpha=\beta=\frac{2}{\rho},\qquad \gamma=\rho,\qquad c=2,\qquad c_1=2,
\]
gives
\[
    E_k\le \frac{4R}{\rho k+2-\rho}.
\]
Finally,
\[
    \norm{\vect{x}_k-T(\vect{x}_k)}
    \le E_k+\delta_{k-1}\norm{T(\vect{x}_k)-g(\vect{x}_k)}
    \le \frac{4R}{\rho k+2-\rho}
       +\frac{4R}{\rho k+2-\rho},
\]
which proves the claim.
\end{proof}

\clearpage
\bibliographystyle{alpha}
\bibliography{ref}

\end{document}